\documentclass[a4paper
,11pt
,leqno
,twoside
]{article}

\usepackage{amsmath, amsthm, amssymb,bm}
\numberwithin{equation}{section}

\newcommand{\bC}{\mathbb{C}}
\newcommand{\bR}{\mathbb{R}}
\newcommand{\bQ}{\mathbb{Q}}

\newcommand{\bE}{\mathbb{E}}

\newcommand{\mf}[1]{\mathfrak{#1}}
\newcommand{\mr}[1]{\mathrm{#1}}
\newcommand{\mcal}[1]{\mathcal{#1}}

\def\tr{\mathrm{tr} \,}
\def\sgn{\mathrm{sgn}}
\def\diag{\mathrm{diag}}
\def\hf{\mathrm{hf}}
\def\Wg{\mathrm{Wg}}
\def\tWg{\widetilde{\mathrm{Wg}}}

\def\Sym{\mathrm{Sym}}

\def\({ \left( }
\def\){ \right)}

\theoremstyle{plain}
\newtheorem{thm}{Theorem} 
\newtheorem{prop}{Proposition}
\newtheorem{lem}[prop]{Lemma}
\newtheorem{cor}[thm]{Corollary}
\theoremstyle{definition}

\newtheorem{example}{Example}
\newtheorem{remark}{Remark}
\theoremstyle{conjecture}

\theoremstyle{problem}

\title{\bfseries General moments of the inverse real Wishart distribution and
orthogonal Weingarten functions
}
\author{Sho Matsumoto}
\date{\empty}

\begin{document}

\maketitle

\begin{abstract}
We study a random positive definite symmetric matrix distributed according to
a real Wishart distribution.
We compute general moments of the random matrix and of its inverse explicitly. 
To do so, we employ the orthogonal Weingarten function,
which was recently introduced in 
the study for Haar-distributed orthogonal matrices.
As applications, we give formulas for moments of traces of a Wishart matrix and its inverse.

\end{abstract}


\section{Introduction}

\subsection{Wishart distributions}

Let $d$ be a positive integer.
Let $\mathrm{Sym}(d)$ be the $\bR$-linear space
of $d \times d$  real symmetric  matrices,
and $\Omega=\Sym^+(d)$ the open convex cone of all positive definite
matrices in $\Sym(d)$.
Let $\sigma =(\sigma_{ij})_{1 \le i,j \le d} \in \Omega$, 
and let 
$$
\beta \in \left\{\frac{1}{2},\frac{2}{2},\dots, \frac{d-1}{2} 
\right\} \sqcup 
\(\frac{d-1}{2}, +\infty\).
$$
Then there exists a probability measure
$\mf{W}_{d,\beta,\sigma}$ on $\Omega$ such that
its moment generating function (or its Laplace transform)  is given by 
$$
\int_{\Omega} e^{\tr (\theta w)} \mf{W}_{d,\beta,\sigma}(\mathrm{d} w) =\det(I_d- \theta\sigma)^{-\beta},
$$
where $\theta$ is any $d \times d$ symmetric matrix such that $\sigma^{-1} -\theta \in
\Omega$.
We call $\mf{W}_{d,\beta,\sigma}$ the {\it real Wishart distribution} on
$\Omega$ with parameters $(\beta,\sigma)$ .

We call a random matrix $W \in \Omega$ a {\it real Wishart matrix}
associated with parameters $(\beta,\sigma)$
and write  $W \sim W_d(\beta,\sigma;\mathbb{R})$
if its distribution is  $\mf{W}_{d,\beta,\sigma}$.
Thus the moment generating function for $W$ is given by 
$$
\mathbb{E} [e^{\tr (\theta W)}] = \det(I_d- \theta\sigma)^{-\beta}, 
$$
with $\theta$ being as above.
Here $\bE$ stands for the average.

If $2\beta$ is a positive integer, $p= 2\beta$ say, 
then a Wishart matrix $W$ is expressed as follows.
Let $X_1,\dots,X_p$ be  $d$-dimensional random column vectors  
distributed independently according to the Gaussian distribution 
$N_d(0,\tfrac{1}{2}\sigma)$. 
Then
$$
W= X_1 X_1^t + \cdots +X_p X_p^t,
$$
where $X_i^t$ is the transpose of $X_i$, i.e., a row vector.

If $\beta > \tfrac{d-1}{2}$ (not necessarily an integer),
the distribution $\mf{W}_{d,\beta,\sigma}$ has the expression
$$
\mf{W}_{d,\beta,\sigma}( \mathrm{d} w)= f(w;d,\beta,\sigma) \mf{L} (\mathrm{d} w),
$$
where $f(w;d,\beta,\sigma)$ is the density function  given by
\begin{equation} \label{eq:Density}
f(w;d,\beta,\sigma)=
\Gamma_d(\beta)^{-1} (\det \sigma)^{-\beta} (\det w)^{\beta-\frac{d+1}{2}} e^{-\tr(\sigma^{-1} w)}
 \qquad (w \in \Omega)
\end{equation}
with the multivariate gamma function
$$
\Gamma_d(\beta)= \pi^{d(d-1)/4} \prod_{j=1}^d \Gamma\(\beta-\frac{1}{2}(j-1)\).
$$
Here  $\mf{L}$ is the Lebesgue measure on $\mr{Sym}(d)$ 
defined by
$$
\mf{L}(\mathrm{d} w)= \prod_{1 \le i \le j \le d} \mathrm{d} w_{ij} \qquad 
\text{with $w=(w_{ij})_{1 \le i,j \le d}$}.
$$

Likewise, a {\it complex} Wishart distribution is defined on
the set of all $d \times d$ positive definite hermitian complex matrices.
Given a Wishart matrix $W$, the distribution of the inverse matrix
$W^{-1}$ is called the {\it inverse} (or {\it inverted})  Wishart distribution.
We denote by $W_{ij}$ and $W^{ij}$ the $(i,j)$-entry of $W$ and $W^{-1}$,
respectively.

The Wishart distributions are fundamental distributions in multivariate statistical analysis.
We refer to \cite{Mu}.
The structure of Wishart distributions has been studied for a long time,
nevertheless, a lot of results are obtained only recently.
We are interested in moments of the forms $\bE[P(W)]$ and $\bE[P(W^{-1})]$, where
$P(A)$ is a polynomial in entries $A_{ij}$ of a matrix $A$.
In particular, we would like to compute {\it general moments} 
$$
\bE[W_{i_1 j_1} W_{i_2 j_2} \cdots W_{i_k j_k}] \qquad \text{and} \qquad
\bE[W^{i_1 j_1} W^{i_2 j_2} \cdots W^{i_k j_k}]
$$
for $W$ and $W^{-1}$, respectively.

Von Rosen \cite{VRosen} computed the general moments
 of low orders for $W^{-1}$.
Lu and Richards \cite{LR} gave formulas for $W$ by applying MacMahon's master theorem.
Graczyk et al. \cite{GLM1} gave  formulas for $W^{\pm 1}$ in the complex case
by using representation theory 
of symmetric groups, while 
they \cite{GLM2} gave results for only $W$ (not $W^{-1}$) in the real case
by using representation theory of hyperoctahedral groups.
Letac and Massam \cite{LM1} computed the moments $\bE[P(W)]$ and $\bE[P(W^{-1})]$
in both real and complex cases, 
where the $P$ is a polynomial depending only on the eigenvalues of a matrix.
Furthermore, a {\it noncentral} Wishart distribution was also studied, see \cite{LM2} and \cite{KN}.

\subsection{Results} \label{subsec:results}

Our main purpose in the present paper is to compute the general moment
$$
\bE [W^{i_1 j_1} W^{i_2 j_2} \cdots W^{i_k j_k}]
$$
for an {\it inverse real} Wishart matrix $W^{-1}=(W^{ij})$.
As we described, in the complex case Graczyk et al. \cite{GLM1} obtained formulas for
such a moment by the represention-theoretic approach.
Our main results are precisely their counterparts for the real case,
which had been unsolved.

To describe our main result, we recall {\it perfect matchings}.
Let $n$ be a positive integer and put $[n]=\{1,2,\dots,n \}$.
A perfect matching $\mf{m}$ on the $2n$-set $[2n]$ is an unordered pairing
of letters $1,2,\dots,2n$.
Denote by $\mcal{M}(2n)$ the set of all such perfect matchings.
For example, $\mcal{M}(4)$ consists of three elements:
$\{ \{1,2\}, \{3,4\} \}$, $\{ \{1,3\}, \{2,4\} \}$, and 
$\{\{1,4\}, \{2,3\}\}$.

Given a perfect matching $\mf{m} \in \mcal{M}(2n)$, we attach 
a (undirected) graph $G=G(\mf{m})$ defined as follows.
The vertex set of $G$ is $[2n]$.
The edge set of $G$ is 
$$
\big\{ \{ 2k-1,2k \} \ | \ k \in [n] \big\} \sqcup
\big\{ \{p,q\} \ | \ \{p,q\} \in \mf{m} \big\}.
$$
Then each vertex has just two edges, and
each connected component of $G$ has even vertices.
We denote by $\kappa(\mf{m})$ 
the number of connected components in $G(\mf{m})$.
For example, given $\mf{m}=
\{\{1,3\},\{2,7\},\{4,8\}, \{5,6\} \} \in \mcal{M}(8)$, 
the graph $G(\mf{m})$ has two connected components 
(where one component has vertices $1,2,3,4,7,8$ and the other has  $5,6$),
and therefore  $\kappa(\mf{m})=2$.

Now we give a formula of the general moments for $W$.
It would be reader-friendly to 
describe the results of both the real Wishart law itself 
(Theorem \ref{thm:GM})
and its inverse (Theorem \ref{thm:InvGM}) on display. 

\begin{thm} \label{thm:GM}
Let $W =(W_{ij})_{1 \le i,j \le d} \sim W_d(\beta,\sigma;\bR)$.
Given indices
$k_1,k_2,\dots,k_{2n}$ 
from $\{1,\dots,d\}$, we have
\begin{equation} \label{eq:thm:GM}
\bE[W_{k_1 k_2} W_{k_3 k_4}\cdots W_{k_{2n-1} k_{2n}}]=2^{-n}\sum_{\mf{m} \in \mcal{M}(2n)}
(2\beta)^{\kappa(\mf{m})} \prod_{ \{p,q\} \in \mf{m} } \sigma_{k_p k_q}. 
\end{equation}
\end{thm}

For example, since 
$\kappa(\{ \{1,2\}, \{3,4\} \})=2$ and
$\kappa(\{ \{1,3\}, \{2,4\} \})=\kappa(\{ \{1,4\}, \{2,3\} \})=1$
we have
\begin{equation} \label{eq:2degreeW}
\bE[W_{k_1 k_2} W_{k_3 k_4}]= \beta^2 \sigma_{k_1 k_2} \sigma_{k_3 k_4}
+ \frac{\beta}{2} \sigma_{k_1 k_3} \sigma_{k_2 k_4} 
+ \frac{\beta}{2} \sigma_{k_1 k_4} \sigma_{k_2 k_3}. 
\end{equation}

Theorem \ref{thm:GM} is not new.
Indeed, it is equivalent to Theorem 10 in \cite{GLM2}.
Moreover, Kuriki and Numata \cite{KN} extended it to noncentral Wishart distributions
very recently.
We revisit it in the framework of {\it alpha-hafnians}.
We develop a theory of the alpha-hafnians in Section \ref{sec:hafnians}, 
and apply it to the proof of Theorem \ref{thm:GM} in Section \ref{sec:ProofThm1}.

The following is our main result.
Let $\sigma^{ij}$ be the $(i,j)$-entry of the inverse matrix $\sigma^{-1}$.

\begin{thm} \label{thm:InvGM}
Let $W \sim W_d(\beta,\sigma;\bR)$.
Put $\gamma=\beta-\frac{d+1}{2}$ and suppose 
$\gamma>n-1$.
Given indices $k_1,k_2,\dots,k_{2n}$  from $\{1,\dots,d\}$, we have
\begin{equation} \label{eq:InvGM2}
\bE[W^{k_1 k_2} W^{k_3 k_4}\cdots 
W^{k_{2n-1} k_{2n}}]=
\sum_{\mf{m} \in \mcal{M}(2n)}
\widetilde{\mr{Wg}}(\mf{m}; \gamma)
\prod_{ \{p,q\} \in \mf{m} }
\sigma^{k_p k_q}.
\end{equation}
Here $\widetilde{\mr{Wg}}(\mf{m}; \gamma)$ 
is defined in Section \ref{sec:ProofThm2} below.
\end{thm}

For example, for $\gamma>1$ we will see
\begin{align*}
\tWg(\{ \{1,2\}, \{3,4\} \};\gamma)=& \frac{2\gamma-1}{\gamma(\gamma-1)(2\gamma+1)}, \\
\tWg(\{ \{1,3\}, \{2,4\} \};\gamma)= 
\tWg(\{ \{1,4\}, \{2,3\} \};\gamma)=& \frac{1}{\gamma(\gamma-1)(2\gamma+1)},
\end{align*}
and therefore we have
$$
\bE[W^{k_1 k_2} W^{k_3 k_4}]= 
\frac{1}{\gamma(\gamma-1)(2\gamma+1)}( (2\gamma-1) \sigma^{k_1 k_2} \sigma^{k_3 k_4}
+  \sigma^{k_1 k_3} \sigma^{k_2 k_4} 
+  \sigma^{k_1 k_4} \sigma^{k_2 k_3}). 
$$

The quantity $\widetilde{\mr{Wg}}(\mf{m}; \gamma)$
is a slight deformation of the {\it orthogonal Weingarten function}.
The function was introduced by Collins and his coauthors \cite{CM, CS},
in order to compute the general moments 
for a Haar-distributed orthogonal matrix.
In general, 
 $\tWg(\mf{m};\gamma)$ \ $(\mf{m} \in \mcal{M}(2n))$ 
is given by a sum over partitions of $n$, and derived from 
the harmonic analysis of the Gelfand pair $(S_{2n}, H_n)$,
where $S_{2n}$ is the symmetric group and $H_n$ is the hyperoctahedral group.
Amazingly,  the same function thus appears in two different random matrix systems.
In Section \ref{sec:Wg},
we review the theory of the Weingarten function developed in
\cite{CM,Mat2}, 
and, in Section \ref{sec:ProofThm2}, 
we prove Theorem \ref{thm:InvGM}.

In Section \ref{sec:Appl}, we give
applications of Theorem \ref{thm:GM} and Theorem \ref{thm:InvGM}.
In particular, we obtain results of Letac and Massam \cite{LM1} as corollaries
of Theorem \ref{thm:GM} and Theorem \ref{thm:InvGM}.

\section{Alpha-hafnians} \label{sec:hafnians}

\subsection{An expansion formula for alpha-hafnians}

Let $A$ be a $2n \times 2n$ symmetric matrix
$A=(A_{pq})_{p,q \in [2n]}$.
Let $\alpha$ be a complex number. 
We define an {\it $\alpha$-hafnian} of $A$ (see \cite{KN2}) by
$$
\hf_\alpha(A)= \sum_{\mf{m} \in \mcal{M}(2n)} \alpha^{\kappa(\mf{m})}
\prod_{\{p,q\} \in \mf{m}} A_{pq}.
$$
The ordinary hafnian of $A$ is nothing but $\hf_1(A)$.
For example, if $n=2$, 
$$
\hf_\alpha(A)= \alpha^2 A_{12} A_{34} + \alpha A_{13} A_{24} + \alpha A_{14} A_{23}.
$$
We remark that $\hf_\alpha(A)$ does not depend on 
diagonal entries $A_{11},A_{22},\dots,A_{2n,2n}$.
Note that the right hand side in \eqref{eq:thm:GM} is equal to $2^{-n} \hf_{2\beta}
(\sigma_{k_p k_q})_{p,q \in [2n]}$.

\begin{prop} \label{prop:Hfrec}
Let $A=(A_{pq})_{p,q \in [2n]}$ be a symmetric matrix.
Let $D=(A_{pq})_{p,q \in [2n-2]}$.
For each $j=1,2,\dots,2n-2$,
let $B^{(j)}$ be the symmetric matrix obtained by replacing the $j$th row/column of $D$
by the $(2n-1)$th row/column of $A$.
In formulas,
$B^{(j)}=(B^{(j)}_{pq})_{p,q \in [2n-2]}$ is given by
$$
B^{(j)}_{pq}= 
\begin{cases}
A_{2n-1,2n-1} & \text{if $p=j$ and $q=j$,} \\
A_{2n-1, q} & \text{if $p=j$ and $q \not=j$,} \\
A_{p,2n-1} & \text{if $p \not=j$ and $q=j$,} \\
A_{p,q} & \text{if $p \not=j$ and $q \not=j$}.
\end{cases}
$$  
Then we have 
\begin{equation} \label{eq:Hfrec}
\hf_\alpha(A)= \sum_{j=1}^{2n-2} A_{j,2n} \hf_{\alpha} (B^{(j)}) + \alpha A_{2n-1,2n} \hf_\alpha(D).
\end{equation}
\end{prop}
We call \eqref{eq:Hfrec} an {\it expansion formula for an $\alpha$-hafnian} with respect to 
the $(2n)$th row/column.

\begin{proof}
For each $j=1,2,\dots,2n-1$, we set
$$
\mcal{M}_j(2n)= \{ \mf{m} \in \mcal{M}(2n) \ | \ \{j,2n\} \in \mf{m} \}.
$$
Then $\mcal{M}(2n) = \bigsqcup_{j=1}^{2n-1} \mcal{M}_j(2n)$.
We define a one-to-one map $\mf{m} \mapsto \mf{n}$ 
from $\mcal{M}_j(2n)$ to $\mcal{M}(2n-2)$ as follows.

First, 
suppose $j=2n-1$. Given $\mf{m} \in \mcal{M}_{2n-1}(2n)$, we let
$\mf{n} \in \mcal{M}(2n-2)$ to be the perfect matching 
obtained from $\mf{m}$ by removing the block $\{2n-1,2n\}$.
It is clear that the mapping 
$\mcal{M}_{2n-1}(2n) \ni \mf{m} \mapsto \mf{n} \in \mcal{M}(2n-2)$ is bijective and 
that  $\kappa(\mf{m})= \kappa(\mf{n})+1$.

Next, suppose $j \in [2n-2]$.
Given $\mf{m} \in \mcal{M}_j(2n)$, we let $\mf{n} \in \mcal{M}(2n-2)$ to be
obtained by removing the block $\{j,2n\}$ and a block $\{i,2n-1\}$ 
(with some $i \in [2n-2]$) and by adding $\{i,j\}$. 
It is easy to see that this mapping 
$\mcal{M}_{j}(2n) \ni \mf{m} \mapsto \mf{n} \in \mcal{M}(2n-2)$ is bijective,  
$\kappa(\mf{m})= \kappa(\mf{n})$, and
$\prod_{ \{p,q\} \in \mf{m}} A_{pq}= A_{j,2n} \prod_{\{p,q\} \in \mf{n}} B_{pq}^{(j)}$.

For example, consider $\mf{m}=\{\{1,4\},  \{2,5\},  \{3,6\} \}$.
Then $\mf{m} \in \mcal{M}_3(6)$, and we obtain $\mf{n}=\{\{1,4\}, \{2,3\}\} \in \mcal{M}(4)$.
Therefore, we have $\kappa (\mf{m})=1=\kappa(\mf{n})$ 
and  $\prod_{\{p,q\} \in \mf{m}} A_{pq}=A_{14}A_{25}A_{36} = A_{36} B^{(3)}_{14} B^{(3)}_{23}
=A_{36} \prod_{\{p,q\} \in \mf{n}} B^{(3)}_{pq}$.

Using the correspondence $\mcal{M}_{j}(2n) \ni \mf{m} \leftrightarrow \mf{n} \in \mcal{M}(2n-2)$
with $j=1,2,\dots,2n-1$,
it follows  that
\begin{align*}
\hf_\alpha(A)=& \sum_{j=1}^{2n-1}  A_{j,2n} \sum_{ \mf{m} \in \mcal{M}_j(2n)} \alpha^{\kappa(\mf{m})}
\prod_{\begin{subarray}{c} \{p,q\} \in \mf{m} \\ \{p,q\} \not=\{j,2n\} \end{subarray}}
A_{pq} \\
=& A_{2n-1,2n} \sum_{\mf{n} \in \mcal{M}(2n-2)} \alpha^{\kappa(\mf{n})+1}
\prod_{\{p,q\} \in \mf{n}} A_{pq} \\
&+ \sum_{j=1}^{2n-2} A_{j,2n} 
\sum_{\mf{n} \in \mcal{M}(2n-2)} \alpha^{\kappa(\mf{n})} 
\prod_{\{p,q\} \in \mf{n}}
 B_{pq}^{(j)},
\end{align*}
which is equal to $A_{2n-1,2n} \alpha \cdot \hf_\alpha(D) 
+\sum_{j=1}^{2n-2}A_{j,2n}  \hf_\alpha(B^{(j)})$.
\end{proof}

\subsection{Another expression for $\alpha$-hafnians}

Let $S_{n}$ be the symmetric group on $[n]$.
Each permutation $\pi$ is uniquely decomposed into a product of cycles.
For example, $\pi=\( \begin{smallmatrix} 1 & 2 & 3 & 4 & 5 & 6 \\
5 & 6 & 1 & 4 & 3 & 2 \end{smallmatrix} \) \in S_6$ is expressed as
$\pi=(1 \to 5 \to 3 \to 1)(2 \to 6 \to 2)(4 \to 4)$. 
Denote by $C(\pi)$ the set of all cycles of $\pi$, and let $\nu(\pi)$ be
the number of cycles of $\pi$: $\nu(\pi)=|C(\pi)|$.

Let $A=(A_{pq})_{p,q \in [2n]}$ be a symmetric matrix.
For each $k,l \in [n]$, we denote by $A[k,l]$ the $2 \times 2$ matrix
$$
A[k,l]= \begin{pmatrix} A_{2k-1,2l-1} & A_{2k-1, 2l} \\ A_{2k,2l-1} & A_{2k,2l} \end{pmatrix}.
$$
For a cycle $c=(c_r \to c_1 \to c_2 \to \dots \to c_r)$ on $\{1,\dots,n\}$, we put
$$
P_{c}(A)=
\tr (A[c_1, c_2] J A[c_2, c_3] J \cdots A[c_r, c_1] J), 
\qquad \text{with $J= \begin{pmatrix} 0 & 1 \\ 1 & 0 \end{pmatrix}$}.
$$
In particular,
$P_{(c_1 \to c_1)}(A)=\tr(A[c_1,c_1] J)=2A_{2c_1-1,2c_1}$
for a $1$-cycle $(c_1 \to c_1)$.
It is easy to see that 
$P_c(A)$ can be written
\begin{equation}
P_c(A)
= \sum_{j_1, j_2,\dots,j_{2r}} A_{j_{2r}, j_1} A_{j_2,j_3} \cdots A_{j_{2r-2},j_{2r-1}} 
\end{equation}
summed over 
$(j_{2k-1},j_{2k}) \in \{(2c_k-1,2c_k), (2c_k,2c_k-1)\}$ \ $(k=1,2,\dots,{r})$. 
For a permutation $\pi \in S_n$, we define
$$
P_\pi(A)= \prod_{ c \in C(\pi)} P_{c}(A).
$$

Similarly, given an $r$-cycle $c=(c_r \to c_1 \to c_2 \to \cdots \to c_r)$,  
we let $c_r$ to be the largest number among $\{c_1,c_2,\dots,c_r\}$.
We define $Q_c(A)$ as follows:
If $r=1$ then $Q_c(A)= A_{2c_1-1,2c_1}$; if $r \ge 2$ then
$$
Q_c(A)= \sum_{(j_1,j_2)} \cdots \sum_{(j_{2r-3}, j_{2r-2})}
A_{2c_r-1, j_1} A_{j_2 j_3} A_{j_4 j_5} \cdots A_{j_{2r-2}, 2c_{r}},
$$
summed over 
$(j_{2k-1},j_{2k}) \in \{(2c_k-1,2c_k), (2c_k,2c_k-1)\}$ \ $(k=1,2,\dots,{r-1})$. 
As $P_\pi(A)$, we define
$$
Q_\pi(A)= \prod_{ c \in C(\pi)} Q_{c}(A).
$$
For example, for a cycle $(3 \to 2 \to 1 \to 3)$,
we have
\begin{align*}
Q_{c}(A)=& \sum_{(j_1,j_2) \in \{ (3,4), (4,3) \}}
\sum_{(j_3,j_4) \in \{ (1,2), (2,1) \}} A_{5 j_1} A_{j_2 j_3} A_{j_4 6} \\
=& A_{53} A_{41} A_{26} + A_{54} A_{31} A_{26} +A_{53} A_{42} A_{16} + A_{54} A_{32} A_{16}.
\end{align*}  

\begin{lem} \label{lem:PQ}
Let $c= (c_r \to c_1 \to c_2 \to \cdots \to c_r)$ be a cycle.
Then
$$
P_c(A)= Q_c(A)+ Q_{c^{-1}}(A),
$$
where $c^{-1}= (c_r \to \cdots \to c_2 \to c_1 \to c_r)$.
\end{lem}

\begin{proof}
Suppose $c_r$ is the largest number in $\{c_1,\dots, c_r\}$.
We can express
$$
P_c(A)= \sum_{j_1,j_2,\dots,j_{2r-2}}
A_{2c_r-1, j_1} A_{j_2,j_3} \cdots A_{j_{2r-2},2c_r} +
\sum_{j_1,j_2,\dots,j_{2r-2}}
A_{2c_r, j_1} A_{j_2,j_3} \cdots A_{j_{2r-2},2c_r-1},
$$
summed over 
$(j_{2k-1},j_{2k}) \in \{(2c_k-1,2c_k), (2c_k,2c_k-1)\}$ \ $(k=1,2,\dots,{r-1})$. 
Here the first sum coincides with $Q_c(A)$, while the second one does with $Q_{c^{-1}}(A)$.
\end{proof}

\begin{prop} \label{prop:AlphaHf}
Let $A=(A_{pq})_{p,q \in [2n]}$ be a symmetric matrix.
Then
$$
\hf_\alpha(A) = 
\sum_{\pi \in S_n} \(\frac{\alpha}{2}\)^{\nu(\pi)} P_{\pi}(A)
=\sum_{\pi \in S_n} \alpha^{\nu(\pi)} Q_{\pi}(A).
$$ 
\end{prop}

This is a key lemma in our proofs of Theorem \ref{thm:GM} and Theorem \ref{thm:InvGM}.
We show this proposition in the next subsection.

\begin{remark}
Let $A=(A_{ij})_{1 \le i,j \le n}$ be a complex matrix
and $\alpha$ a complex number.
An {\it $\alpha$-permanent} of $A$ is defined by
$$
\mr{per}_\alpha (A)= \sum_{\pi \in S_n} \alpha^{\nu(\pi)} \prod_{i=1}^n A_{i \pi(i)}.
$$
It intertwines the permanent and determinant:
$$
\mr{per}_1(A)= \mr{per}(A) = \sum_{\pi \in S_n} \prod_{i=1}^n A_{i \pi(i)}
\qquad \text{and} \qquad
\mr{per}_{-1}(A)= (-1)^n \det(A).
$$
It is also called an {\it $\alpha$-determinant}.
See \cite{VJ} and also \cite{Shirai}.
Alpha-hafnians are generalizations of the alpha-permanents in the following sense.
Given a matrix $A=(A_{ij})_{1 \le i,j \le n}$, we define the $2n \times 2n$ symmetric matrix
$B=(B_{pq})_{1 \le p,q \le 2n}$ by 
$$
B_{2i-1,2j-1}=B_{2i,2j}=0 \qquad \text{and} \qquad
 B_{2i-1,2j}=B_{2j,2i-1}=A_{ij} \qquad \text{for all 
$i,j=1,2,\dots,n$}.
$$
Then, since $Q_c(B)=A_{c_r, c_1}A_{c_1, c_2} \dots A_{c_{r-1}, c_r}$ for 
$c= (c_r \to c_1 \to c_2 \to \cdots \to c_r)$, 
it follows from Proposition \ref{prop:AlphaHf} that $\hf_\alpha(B)=\mr{per}_\alpha(A)$.
Thus any $\alpha$-permanent can be given by an $\alpha$-hafnian.
\end{remark}

\begin{remark}
Let $B=(B_{pq})_{p,q \in [2n]}$ be a skew-symmetric matrix and let 
$\alpha$ be a complex number.
In \cite{Mat1}, an $\alpha$-pfaffian of $B$ was defined.
In a similar way to the proof of Proposition  \ref{prop:AlphaHf},
we can see that the definition in \cite{Mat1} is equivalent to the expression
$$
\mr{pf}_\alpha(B)=\sum_{\mf{m} \in \mcal{M}(2n)} (-\alpha)^{\kappa(\mf{m})}
\sgn(\mf{m}) \prod_{\{p,q\} \in \mcal{M}(2n)} B_{pq}.
$$
Here, for $\mf{m}=\{ \{\mf{m}(1), \mf{m}(2)\},\dots, \{\mf{m}(2n-1),\mf{m}(2n)\}\}$
we define 
$$
\sgn(\mf{m}) \prod_{\{p,q\} \in \mcal{M}(2n)} B_{pq}
=\sgn \begin{pmatrix} 1 & 2& \cdots &2n \\ \mf{m}(1) & \mf{m}(2) & \cdots & \mf{m}(2n) 
\end{pmatrix} \cdot
B_{\mf{m}(1) \mf{m}(2)} \cdots B_{\mf{m}(2n-1) \mf{m}(2n)}.
$$
When $\alpha=-1$, the $\alpha$-pfaffian is exactly the ordinary pfaffian.
Moreover, as $\alpha$-hafnians are so, 
the $\alpha$-pfaffians are generalizations of $\alpha$-permanents.
\end{remark}

\subsection{Proof of Proposition \ref{prop:AlphaHf}}

Put
$$
\widetilde{\hf}_\alpha(A)=\sum_{\pi \in S_n} \(\frac{\alpha}{2}\)^{\nu(\pi)} P_{\pi}(A)
=\sum_{\pi \in S_n} \alpha^{\nu(\pi)} Q_{\pi}(A)
$$
for any $n \ge 1$ and any symmetric matrix $A$ of size $2n$.
Here the second equality follows from Lemma \ref{lem:PQ}.

Let $B^{(1)}, B^{(2)},\dots,B^{(2n-2)},D$ be as in Proposition \ref{prop:Hfrec}.
In order to obtain  Proposition \ref{prop:AlphaHf},
it is enough to show the recurrence formula
\begin{equation} \label{eq:Hfrectilde}
\widetilde{\hf}_\alpha(A)
=\sum_{j=1}^{2n-2} A_{j,2n} \widetilde{\hf}_{\alpha} (B^{(j)}) + \alpha A_{2n-1,2n} 
\widetilde{\hf}_\alpha(D).
\end{equation}

To see \eqref{eq:Hfrectilde},
we will show a recurrence formula involving $Q_c(A)$ and $P_c(A)$.
For each $k \in [n]$, we denote by 
$S_n^{(k)}$ the subset of permutations in $S_n$ such that $\pi(k)=n$. 
Note $S_n= \bigsqcup_{k=1}^n S_n^{(k)}$.

Let $k \in [n-1]$ and let $\pi \in S_n^{(k)}$. 
Let $u_n(\pi) \in C(\pi)$ be the cycle including the letter $n$, which is of the form
$$
u_n(\pi)=(n \to c_1 \to c_2 \to \cdots \to c_r \to k \to n),
$$
with  (possibly empty) distinct $c_1,\dots,c_r \in [n] \setminus\{k,n\}$. 
Then, define
$$
\tilde{u}_n(\pi)= (n \to c_r \to \cdots \to c_2 \to c_1 \to k \to n),
$$
and let $\tilde{\pi}$ be the permutation obtained by replacing 
$u_n(\pi)$ in $\pi$ by $\tilde{u}_n(\pi)$.
Note that $u_n(\tilde{\pi})=\tilde{u}_n(\pi)$, and that
$\tilde{\pi}=\pi$ if and only if $u_n(\pi)$ is a $2$ or $3$-cycle.
The map $\pi \mapsto \tilde{\pi}$ is an involution on  $S_n^{(k)}$.
For example, given $\pi=(7 \to 3 \to 1 \to 2 \to 7)(6 \to 4 \to 6)(5 \to 5) \in S_7$, 
we have  
$\tilde{\pi}=(7 \to 1 \to 3 \to 2 \to 7)(6 \to 4 \to 6)(5 \to 5)$.

In general, for the cycle $u_n(\pi)=(n \to c_1 \to c_2 \to \cdots \to c_r \to k \to n)$ 
with $k \not=n$,
we see that
\begin{align*}
Q_{u_n(\pi)}(A)=& \sum_{(j_1,j_2)} \cdots \sum_{(j_{2r-1},j_{2r})}  
(A_{2n-1, j_1} A_{j_2,j_3} \cdots A_{j_{2r}, 2k-1} A_{2k,2n} +
A_{2n-1, j_1} A_{j_2,j_3} \cdots A_{j_{2r}, 2k} A_{2k-1,2n}) \\
=& \sum_{(j_1,j_2)} \cdots \sum_{(j_{2r-1},j_{2r})}
(B^{(2k)}_{2k, j_1} B^{(2k)}_{j_2,j_3} \cdots B^{(2k)}_{j_{2r},2k-1} A_{2k,2n}
+B^{(2k-1)}_{2k-1, j_1} B^{(2k-1)}_{j_2,j_3} \cdots B^{(2k-1)}_{j_{2r},2k} A_{2k-1,2n})
\end{align*}
summed over 
$$
(j_{2p-1},j_{2p}) \in \{(2c_p-1,2c_p), (2c_p,2c_p-1)\} \qquad (p=1,2,\dots,{r}).
$$
Similarly,
\begin{align*}
Q_{\tilde{u}_n(\pi)}(A)=& 
\sum_{(j_1,j_2)} \cdots \sum_{(j_{2r-1},j_{2r})}  
(A_{2n-1, j_{2r}} \cdots A_{j_3, j_2} A_{j_1,2k-1} A_{2k,2n} +
A_{2n-1, j_{2r}} \cdots A_{j_3, j_2} A_{j_1,2k} A_{2k-1,2n}) \\
=&
\sum_{(j_1,j_2)} \cdots \sum_{(j_{2r-1},j_{2r})}  
(B_{2k, j_{2r}}^{(2k)} \cdots B^{(2k)}_{j_3, j_2} B^{(2k)}_{j_1,2k-1} A_{2k,2n} +
B^{(2k-1)}_{2k-1, j_{2r}} \cdots B^{(2k-1)}_{j_3, j_2} B^{(2k-1)}_{j_1,2k} A_{2k-1,2n}).
\end{align*}
Therefore, we have
\begin{equation} \label{eq:QrecP2}
Q_{u_n(\pi)}(A)+Q_{u_n(\tilde{\pi})}(A)= A_{2k,2n} 
P_{u_n'(\pi)}(B^{(2k)}) +A_{2k-1,2n}P_{u_n'(\pi)}(B^{(2k-1)}).
\end{equation}
Here $u_n'(\pi)$ is the cycle obtained from $u_n(\pi)$ by removing the letter $n$:
$u_n'(\pi)=(k \to c_1 \to c_2 \to \cdots \to c_r \to k)$.
We note that
the mapping $\pi \mapsto
\pi' :=u_n'(\pi) \prod_{c \in C(\pi) \setminus \{u_n(\pi)\}} c$
is the bijective map from $S_n^{(k)}$ to $S_{n-1}$,
and that $\nu(\pi)=\nu(\pi')$.

Now we go back to the proof of \eqref{eq:Hfrectilde}.
We rewrite 
$$
\widetilde{\hf}_\alpha(A)=
\sum_{\pi \in S_n^{(n)}} \alpha^{\nu(\pi)} Q_\pi(A) +
\sum_{k=1}^{n-1}\sum_{\pi \in S_n^{(k)}} \alpha^{\nu(\pi)} Q_{u_n(\pi)}(A) 
2^{-(\nu(\pi)-1)}\prod_{c \in C(\pi) \setminus 
\{u_n(\pi)\} } P_c(A).
$$
The first sum is equal to
$$
\sum_{\pi' \in S_n} \alpha^{\nu(\pi')+1} Q_{\pi'}(A) Q_{(n)}(A)= \alpha A_{2n-1,2n} 
\hf_\alpha(D)
$$
by a natural bijective map $S_n^{(n)} \to S_{n-1}$,
while, since the map $\pi \mapsto \tilde{\pi}$ is bijective on each $S_{2n}^{(k)}$, 
the terms corresponding to $k \in [n-1]$ in the second sum are equal to
\begin{align*}
& 
\sum_{\pi \in S_n^{(k)} } \(\frac{\alpha}{2}\)^{\nu(\pi)} 
(Q_{u_n(\pi)}(A)+Q_{u_n(\tilde{\pi})}(A))
\prod_{c \in C(\pi) \setminus \{u_n(\pi)\}}
P_c(A)  \\
=&  \sum_{\pi \in S_n^{(k)}} \(\frac{\alpha}{2}\)^{\nu(\pi)}  
(A_{2k,2n} P_{u_n'(\pi)}(B^{(2k)}) +A_{2k-1,2n} P_{u_n'(\pi)}(B^{(2k-1)}))
\prod_{c \in C(\pi) \setminus \{u_n(\pi)\}}
P_c(A) \\
=&   \sum_{\pi' \in S_{n-1}} 
\(\frac{\alpha}{2}\)^{\nu(\pi')}  
(A_{2k,2n} P_{\pi'}(B^{(2k)}) +A_{2k-1,2n} P_{\pi'}(B^{(2k-1)})) \\
=&  A_{2k,2n} \widetilde{\hf}_\alpha(B^{(2k)})+ 
A_{2k-1,2n} \widetilde{\hf}_\alpha(B^{(2k-1)}).
\end{align*}
Here the first equality follows by \eqref{eq:QrecP2}, and
the second equality follows from the bijection $S_n^{(k)} \ni \pi \mapsto
\pi' =u_n'(\pi) \prod_{c \in C(\pi) \setminus \{u_n(\pi)\}} c \in S_{n-1}$.
Hence \eqref{eq:Hfrectilde} follows, and we end the proof of 
Proposition \ref{prop:AlphaHf}.

\section{Proof of Theorem \ref{thm:GM}} \label{sec:ProofThm1}

Let $m_1,\dots,m_n$ and $x$ be $d \times d$ matrices. 
Given a cycle $c=(c_r \to c_1 \to c_2 \to \cdots \to c_r)$ on $[n]$,
we define 
$$
R_c(x;m_1,\dots,m_n) = \tr \( x m_{c_1} x m_{c_2} \cdots x m_{c_r}\).
$$
More generally, for a permutation $\pi \in S_n$, we define
$$
R_\pi(x;m_1,\dots,m_n)= \prod_{c \in C(\pi)} R_c(x;m_1,\dots,m_n).
$$
For example, if $n=6$ and $\pi=(1 \to 5 \to 3  \to 1)(2 \to 6 \to 2)(4 \to 4)$, then
$$
R_\pi(x;m_1,m_2,m_3,m_4,m_5,m_6) = \tr(x m_1 x m_5 x m_3) \tr(x m_2 x m_6) \tr(x m_4).
$$

The following proposition, given in  \cite{GLM2},
is our starting point for the proof of Theorem\,\ref{thm:GM}.
Let $d, \beta, \sigma$ be as in Introduction.

\begin{prop} \label{prop:1st}
Let $W \sim W_d(\beta,\sigma;\bR)$ and let $s_1,\dots,s_n \in \Sym(d)$.
Then
$$
\bE [ \tr(W s_1) \tr (W s_2) \cdots \tr (W s_n) ] = \sum_{\pi \in S_n} \beta^{\nu(\pi)}
R_{\pi}(\sigma;s_1,\dots,s_n).
$$
\end{prop}

\begin{proof}
See Proposition 1 in \cite{GLM2}. See also Theorem 1 in \cite{LM1}.
\end{proof}

Theorem \ref{thm:GM} is a consequence of Proposition \ref{prop:1st}
and Proposition \ref{prop:AlphaHf}.
For $1 \le a, b \le d$, denote by $E_{ab}=E_{ab}^{(d)}$ the matrix unit of size $d$,
whose $(i,j)$-entry is
$(E_{ab})_{ij} = \delta_{ai} \delta_{bj}$.
We apply Proposition \ref{prop:1st} with $s_j= (E_{k_{2j-1} k_{2j}} +E_{k_{2j} k_{2j-1}})/2$
$(1 \le j \le n)$.
Since $W$ is symmetric, we have
$\tr(W s_j)= (W_{k_{2j-1} k_{2j}}+
W_{k_{2j} k_{2j-1}})/2= W_{k_{2j-1} k_{2j}}$, and therefore it follows from 
Proposition \ref{prop:1st}  that
\begin{align*}
&\bE[W_{k_1 k_2} W_{k_3 k_4}\cdots W_{k_{2n-1} k_{2n}}] \\
=&
 2^{-n}\sum_{\pi \in S_n} \beta^{\nu(\pi)} R_\pi(\sigma;
E_{k_1 k_2}+E_{k_2 k_1}, \dots, E_{k_{2n-1} k_{2n}} + E_{k_{2n} k_{2n-1}}).
\end{align*}
From Proposition \ref{prop:AlphaHf},
in order to prove Theorem \ref{thm:GM},
it is sufficient to show
\begin{equation} \label{eq:R_P}
R_\pi(\sigma;
E_{k_1 k_2}+E_{k_2 k_1}, \dots, E_{k_{2n-1} k_{2n}} + E_{k_{2n} k_{2n-1}})
=  P_\pi \( (\sigma_{k_p k_q})_{p,q \in [2n]} \)
\end{equation}
for any permutation $\pi \in S_n$.

To show \eqref{eq:R_P}, let $A=(A_{pq})_{p,q \in [2n]}$ be a symmetric matrix
and let $c=(c_r \to c_1 \to c_2 \to \cdots \to c_r)$ be a cycle.
Equation \eqref{eq:R_P} follows from 
\begin{equation} \label{eq:R_P2}
\tr( A (E_{2c_1-1, 2c_1} + E_{2c_1, 2c_1-1}) 
\cdots A (E_{2c_r-1, 2c_r} + E_{2c_r-1, 2c_r})) = P_c (A),
\end{equation}
with $A=(\sigma_{k_p k_q})_{p,q \in [2n]}$.
Here the $E_{ab}=E_{ab}^{(2n)}$ are $2n \times 2n$ matrix units.
However we may show \eqref{eq:R_P2} as follows: 
\begin{align*}
& \tr( A (E_{2c_1-1, 2c_1} + E_{2c_1, 2c_1-1}) 
\cdots A (E_{2c_r-1, 2c_r} + E_{2c_r-1, 2c_r})) \\
=& \sum_{j_1,j_2,\dots,j_{2r}=1}^{2n} A_{j_{2r} j_1} (E_{2c_1-1, 2c_1} + E_{2c_1, 2c_1-1})_{j_1 j_2} A_{j_2 j_3}\cdots 
A_{j_{2r-2} j_{2r-1}} (E_{2c_r-1, 2c_r} + E_{2c_r, 2c_r-1})_{j_{2r-1} j_{2r}} \\
=& \sum_{j_1,\dots,j_{2r}} A_{j_{2r} j_1} A_{j_2 j_3} \cdots A_{j_{2r-2} j_{2r-1}}.
\end{align*}
Here the last sum is over
$(j_{2k-1},j_{2k}) \in \{(2c_k-1,2c_k), (2c_k,2c_k-1)\}$ \ 
$(k=1,2,\dots,r )$.
Hence we obtain
\eqref{eq:R_P2} and therefore \eqref{eq:R_P}.
It ends the proof of Theorem \ref{thm:GM}.

\section{Orthogonal Weingarten functions} \label{sec:Wg}

We review the theory of the Weingarten function for orthogonal groups;
see \cite{CM,Mat2} for details.
Claims in Subsections \ref{subsec:Hyper}--\ref{subsec:zonal} are also
seen in \cite[VII-2]{Mac}.

\subsection{Hyperoctahedral groups and perfect matchings} \label{subsec:Hyper}

Let $H_n$ be the subgroup in $S_{2n}$ generated by
transpositions $(2k-1 \to 2k \to 2k-1 )$ $(1 \le k \le n)$ and 
by double transpositions $(2i-1 \to 2j-1 \to 2i-1) \cdot 
(2i \to 2j \to 2i)$ \ $(1 \le i<j \le n)$.
The group $H_n$ is called the {\it hyperoctahedral group}.
Note that $|H_n|=2^n n!$.

We embed the set $\mcal{M}(2n)$ into $S_{2n}$ via the mapping
$$
\mcal{M}(2n) \ni \mf{m} \mapsto \begin{pmatrix}
1 & 2 & 3 & 4 & \cdots & 2n \\ 
\mf{m}(1) & \mf{m}(2) & \mf{m}(3) & \mf{m}(4) & \cdots & \mf{m}(2n) 
\end{pmatrix} \in S_{2n}
$$
where  $(\mf{m}(1), \dots, \mf{m}(2n))$ is the unique sequence  satisfying
\begin{align*}
&\mf{m}=\left\{\{\mf{m}(1),\mf{m}(2) \}, \dots, \{\mf{m}(2n-1), \mf{m}(2n)\} \right\}, \\
& \mf{m}(2k-1) < \mf{m}(2k)  \quad (1 \le k \le n), \qquad \text{and} \qquad
 1=\mf{m}(1) < \mf{m}(3) < \cdots <\mf{m}(2n-1).
\end{align*}
The $\mf{m} \in \mcal{M}(2n)$ are representatives of the cosets $g H_n$ of $H_n$ in $S_{2n}$:
\begin{equation} \label{eq:LeftDecom}
S_{2n} = \bigsqcup_{\mf{m} \in \mcal{M}(2n)} \mf{m} H_n.
\end{equation}

\subsection{Coset-types}

A {\it partition} $\lambda=(\lambda_1,\lambda_2,\dots)$ is a weakly decreasing sequence
of nonnegative integers such that $|\lambda|:=\sum_{i \ge 1} \lambda_i$ is finite.
If $|\lambda|=n$, we call $\lambda$ a {\it partition of $n$} and write $\lambda \vdash n$.
Define the length $\ell(\lambda)$ of $\lambda$ by the number of nonzero $\lambda_i$.

Given $g \in S_{2n}$, we attach a graph $G(g)$ with vertices
$1,2,\dots,2n $ and with the edge set
$$
\big\{ \{2k-1,2k\} \ | \ k \in [n] \big\} \sqcup 
\big\{ \{ g(2k-1), g(2k)\} \ | \ k \in [n] \big\}.
$$
Each connected component of $G(g)$ has even vertices.
Let $2\lambda_1,2\lambda_2,\dots, 2\lambda_l$ be numbers of vertices of components.
We may suppose $\lambda_1 \ge \lambda_2 \ge \cdots \ge \lambda_l$.
Then the sequence $\lambda=(\lambda_1,\lambda_2,\dots,\lambda_l)$ is a partition of $n$.
We call the $\lambda$ the {\it coset-type} of $g \in S_{2n}$.

For example, the coset-type of 
$\( \begin{smallmatrix} 1 & 2 & 3 & 4 & 5 & 6 & 7 & 8 \\ 7 & 1 & 6 & 3 & 2 & 8 & 4 & 5 \end{smallmatrix} \)$ in $S_{8}$ is $(2,2)$.

In general, given $g,g' \in S_{2n}$, their coset-types coincide if and only if
$H_n g H_n= H_n g' H_n$.
Hence we have the double coset decomposition of $H_n$ in $S_{2n}$:
\begin{equation} \label{eq:doubledecomposition}
S_{2n}= \bigsqcup_{\rho \vdash n} H_\rho, \qquad \text{where $
H_\rho= \{g \in S_{2n} \ | \ \text{the coset-type of $g$ is $\rho$}\} $}.
\end{equation}
Note $H_{(1^n)}=H_n$ and 
$|H_\rho|= (2^n n!)^2/ (2^{\ell(\rho)} z_\rho)$. 
Here
\begin{equation} \label{eq:z}
z_\rho= \prod_{r \ge 1} r^{m_r(\rho)} m_r(\rho)!
\end{equation}
with multiplicities $m_r(\rho)= | \{ i \ge 1 \ | \ \rho_i=r\}|$ of $r$ in $\rho$.

For $g \in S_{2n}$, denote by $\kappa(g)$ the number of connected components of $G(g)$.
Equivalently, $\kappa(g)$ is the length of the coset-type of $g$.
Under the embedding $\mcal{M}(2n) \subset S_{2n}$, we may define $G(\mf{m})$ and
$\kappa(\mf{m})$ for each $\mf{m} \in \mcal{M}(2n)$.
They are compatible with
their definitions in Subsection \ref{subsec:results}.

\subsection{Zonal spherical functions} \label{subsec:zonal}

For two functions $f_1,f_2$ on $S_{2n}$,
their convolution $f_1 * f_2$ is defined by
$$
(f_1* f_2)(g)= \sum_{g' \in S_{2n}} f_1(g (g')^{-1}) f_2(g') \qquad (g \in S_{2n}).
$$

Let $\mcal{H}_n$ be the set of all complex-valued $H_n$-biinvariant functions on $S_{2n}$:
$$
\mcal{H}_n=\{ f: S_{2n} \to \bC \ | \ f(\zeta g)= f(g \zeta)=f(g) \ (g \in S_{2n}, \ 
\zeta \in H_n)\}.
$$
It is known  that
this is a commutative algebra under convolution, with
unit $\mathbf{1}_{\mcal{H}_n}$ given by
\begin{equation} \label{eq:unitHn}
\mathbf{1}_{\mcal{H}_n}(g) = 
\begin{cases}
(2^n n!)^{-1} & \text{if $g \in H_n$,} \\
0 & \text{otherwise}.
\end{cases}
\end{equation}
Therefore $(S_{2n},H_n)$ is a {\it Gelfand pair} in the sense of \cite[VII.1]{Mac}.
The algebra $\mcal{H}_n$ is called 
the {\it Hecke algebra} associated with the Gelfand pair $(S_{2n},H_n)$.

For each $\lambda \vdash n$ we define
the {\it zonal spherical function} $\omega^\lambda$ by
$$
\omega^\lambda(g)= \frac{1}{2^n n!} \sum_{\zeta \in H_n} \chi^{2\lambda}(g \zeta)
\qquad
(g \in S_{2n}).
$$
Here $\chi^{2\lambda}$ is the irreducible character of $S_{2n}$ associated with
$2\lambda=(2\lambda_1,2\lambda_2,\dots)$.
The $\omega^\lambda$ \ $(\lambda \vdash n)$ form a basis of $\mcal{H}_n$
and have the property
\begin{equation} \label{eq:zonalorthogonal}
\omega^\lambda * \omega^\mu= \delta_{\lambda \mu}\frac{(2n)!}{f^{2\lambda}}
\omega^\lambda \qquad \text{for all $\lambda,\mu \vdash n$}.
\end{equation}
Here $f^{2\lambda}$ is the value of $\chi^{2\lambda}$ at the identity of $S_{2n}$,
or equivalently 
the dimension of the irreducible representation of character  $\chi^{2\lambda}$.
We denote by $\omega^\lambda_\rho$ the value of $\omega^\lambda$
at the double coset $H_\rho$. Note $\omega^\lambda_{(1^n)}=1$ for all $\lambda \vdash n$.

\subsection{Zonal polynomials} \label{subsec:zonal}

We now need the theory of symmetric functions.
Let $\Lambda$ be the algebra of symmetric functions 
in infinitely-many variables $x_1,x_2,\dots$ and
with coefficients in $\bQ$.
Let $\lambda=(\lambda_1,\lambda_2,\dots)$ be a partition of $n$. 
We denote by $p_\lambda$ 
the {\it power-sum symmetric function}:
$$
p_\lambda= \prod_{i=1}^{\ell(\lambda)} p_{\lambda_i} \qquad  \text{and} \qquad
p_k(x_1,x_2,\dots)= x_1^k + x_2^k+ \cdots.
$$
Let $Z_\lambda$ be the {\it zonal polynomial} (or zonal symmetric function):
\begin{equation} \label{eq:Zp}
Z_\lambda= 2^n n! \sum_{\rho \vdash n} 2^{-\ell(\rho)} z_\rho^{-1} \omega^\lambda_\rho p_\rho. 
\end{equation}
Here $z_\rho$ is the quantity defined in \eqref{eq:z}.
Alternatively, for $\rho \vdash n$, 
\begin{equation} \label{eq:pZ}
p_\rho= \frac{2^n n!}{(2n)!} \sum_{\lambda \vdash n} f^{2\lambda} \omega^\lambda_\rho Z_\lambda.
\end{equation}

Recall that $\Lambda$ is the algebra generated by $\{p_r \ | \ r \ge 1\}$ and 
that the $p_r$ are algebraically independent.
Let $z$ be a complex number and let $\phi_z: \Lambda \to \bC$ be the algebra homomorphism
defined by $\phi_z(p_r)= z$ for all $r \ge 1$. 
Then we have the {\it specializations}
\begin{equation} \label{eq:secialC}
\phi_z(p_\rho)= z^{\ell(\rho)} \qquad \text{and} \qquad \phi_z(Z_\lambda)= C_\lambda(z) :=
\prod_{(i,j) \in \lambda} (z+2j-i-1)
\end{equation}
where the product $\prod_{(i,j) \in \lambda}$ stands for
$\prod_{i=1}^{\ell(\lambda)} \prod_{j=1}^{\lambda_i}$, which is
over all boxes of the Young diagram of $\lambda$.
It follows by \eqref{eq:Zp} and \eqref{eq:pZ} that
\begin{equation} \label{eq:specialZp}
C_\lambda(z)=
2^n n! \sum_{\rho \vdash n} 2^{-\ell(\rho)} z_\rho^{-1} \omega^\lambda_\rho z^{\ell(\rho)}
\qquad \text{and} \qquad
z^{\ell(\rho)}= \frac{2^n n!}{(2n)!} \sum_{\lambda \vdash n} f^{2\lambda} \omega^\lambda_\rho
C_\lambda(z).
\end{equation}

\subsection{Weingarten functions} \label{subsec:Wg}

Let $z$ be a complex number such that $C_\lambda(z) \not=0$ for all $\lambda \vdash n$.
We define a function $\mr{Wg}^O(\cdot;z)$ in $\mcal{H}_n$ by
\begin{equation} \label{eq:DefWg}
\mr{Wg}^O(g;z)= \frac{1}{(2n-1)!!} \sum_{\lambda \vdash n}
\frac{f^{2\lambda}}{C_\lambda(z)} \omega^\lambda(g) \qquad (g \in S_{2n}).
\end{equation} 
We call it the {\it orthogonal Weingarten function}
(or {\it Weingarten function for orthogonal groups}).

The function $g \mapsto \mr{Wg}^O(g;z)$ is constant at each double coset $H_\rho$
$(\rho \vdash n)$. 
We denote by (the same symbol) $\mr{Wg}^O(\rho;z)$ its value at $H_\rho$.

\begin{example} \label{ex:Wgz}
\begin{align*}
\mr{Wg}^O((1);z)=& \frac{1}{z}. \\
\mr{Wg}^O((2);z)=& \frac{-1}{z(z+2)(z-1)}. \qquad 
\mr{Wg}^O((1^2);z)= \frac{z+1}{z(z+2)(z-1)}. 
\end{align*}
The list of $\mr{Wg}^O(\rho;z)$ for $|\rho| \le 6$ is seen in \cite{CM}.
\end{example}

Define the function $\mr{G}^O(\cdot ;z)$ in $\mcal{H}_{n}$ by
$$
\mr{G}^O(g;z)= z^{\kappa(g)} \qquad (g \in S_{2n}).
$$
The following lemma is a key in our proof of Theorem \ref{thm:InvGM}.

\begin{lem}[\cite{CM}] \label{lem:Gwg}
$$
\mr{G}^O( \cdot ;z) * \mr{Wg}^O( \cdot;z) = (2^n n!)^2 \mathbf{1}_{\mcal{H}_n}.
$$
Here $\mathbf{1}_{\mcal{H}_n}$ is defined in \eqref{eq:unitHn}.
\end{lem}

\begin{proof}
Recall that if $\rho$ is the coset-type of $g$, then $\kappa(g)=\ell(\rho)$.
From the second formula in \eqref{eq:specialZp}, we have
\begin{equation} \label{eq:ExpandG}
G^O(\cdot;z)= \frac{2^n n!}{(2n)!} \sum_{\lambda \vdash n} f^{2\lambda} C_\lambda(z) 
\omega^\lambda,
\end{equation}
so that
$$
\mr{G}^O( \cdot ;z) * \mr{Wg}^O( \cdot;z)
=  \frac{(2^n n!)^2}{(2n)!} \sum_{\lambda \vdash n} f^{2\lambda} \omega^\lambda
$$
by \eqref{eq:DefWg} and \eqref{eq:zonalorthogonal}.

On the other hand, since $\lim_{t \in \bR, \ t \to +\infty} t^{-n} C_\lambda(t)=1$, 
using the second formula in \eqref{eq:specialZp} again, we may see that
$$
\frac{2^n n!}{(2n)!} \sum_{\lambda \vdash n} f^{2\lambda} \omega^\lambda(g)
= \lim_{t \to +\infty} t^{-n} \frac{2^n n!}{(2n)!} \sum_{\lambda \vdash n} f^{2\lambda}
C_\lambda(t) \omega^\lambda(g) 
= \lim_{t \to +\infty} t^{-(n-\kappa(g))},
$$
which is equal to $1$ if $g \in H_n$, or to zero otherwise.
Hence we have
$$
\mathbf{1}_{\mcal{H}_n} = \frac{1}{(2n)!} \sum_{\lambda \vdash n} f^{2\lambda} \omega^\lambda.
$$
This finishes the proof.
\end{proof}

\subsection{Weingarten calculus for orthogonal groups}

The content in this subsection will not be used in the latter sections.
We here review how the Weingarten function $\mr{Wg}^O$ appears
in the theory of random orthogonal matrices.

Let $O(N)$ be the compact Lie group of $N \times N$ real orthogonal matrices.
The group $O(N)$ is equipped with the  {\it Haar probability measure} $\mathrm{d} O$ such that
$\mathrm{d}(U_1 O U_2) = \mathrm{d} O$ for fixed $U_1,U_2 \in O(N)$ and that
$\int_{O(N)} \mathrm{d} O =1$.

Let $O=(O_{ij})_{i,j \in [N]}$ be a Haar-distributed orthogonal matrix.
Consider a general moment
$$
\bE[O_{i_1 j_1} O_{i_2 j_2} \cdots O_{i_k j_k}] \qquad
(i_1,i_2,\dots,i_k,j_1,j_2,\dots,j_k \in [N]).
$$
From the biinvariant property for the Haar measure,
we can see immediately that
$\bE[O_{i_1 j_1} O_{i_2 j_2} \cdots O_{i_k j_k}]=0$ if $k$ is odd.

\begin{prop}[\cite{CM,CS}] \label{prop:CM}
Let $i_1,\dots,i_{2n}, j_1,\dots,j_{2n}$ be indices in $[N]$.
Assume that $N \ge n$ and let $O=(O_{ij})_{i,j \in [N]}$ be 
a Haar-distributed orthogonal matrix. Then we have
$$
\bE[O_{i_1 j_1} O_{i_2 j_2} \cdots O_{i_{2n} j_{2n}}]
= \sum_{\mf{m},\mf{n} \in \mcal{M}(2n)} \mr{Wg}^O(\mf{m}^{-1} \mf{n};N)
\( \prod_{\{p,q\} \in \mf{m}} \delta_{i_p,i_q} \)
\( \prod_{\{p,q\} \in \mf{n}} \delta_{j_p,j_q} \).
$$
Here each $\mf{m} \in \mcal{M}(2n)$ is regarded as a permutation in $S_{2n}$.
\end{prop}

In particular, using Example \ref{ex:Wgz}, we have
$$
\bE[O_{1, j_1} O_{1,j_2} O_{2,j_3} O_{2,j_4}] =
\frac{1}{N(N+2)(N-1)} ((N+1) \delta_{j_1 j_2} \delta_{j_3 j_4} -
\delta_{j_1 j_3} \delta_{j_2 j_4} - \delta_{j_1 j_4} \delta_{j_2 j_3})
$$
for $N \ge 2$ and $j_1,j_2,j_3,j_4 \in [N]$.

\begin{remark}
Proposition \ref{prop:CM} was first proved in \cite{CS} with a function 
$\mr{Wg}^O$, which was implicitly defined via the equation of 
Lemma \ref{lem:Gwg}.
The explicit expression \eqref{eq:DefWg} was first given in \cite{CM}.
Zinn--Justin \cite{ZJ} (see also \cite{Mat2}) gave
another expression, involving Jucys--Murphy elements. 
\end{remark}

\begin{remark}
If $\ell(\lambda)>N$ then $C_\lambda(N)=0$, and therefore 
the definition \eqref{eq:DefWg} does not make sense unless unless $N \ge n$. 
For $z=N \in \{1,2,\dots,n-1\}$
we extend the definition of the Weingarten function by
$$
\mr{Wg}^O(g;N)= 
\frac{1}{(2n-1)!!} \sum_{
\begin{subarray}{c} \lambda \vdash n \\ \ell(\lambda) \le N \end{subarray}}
\frac{f^{2\lambda}}{C_\lambda(N)} \omega^\lambda(g) \qquad (g \in S_{2n}).
$$ 
Then $\mr{Wg}^O(g;N)$ {\it does} make sense for all $g \in S_{2n}$, and 
Proposition \ref{prop:CM} holds true without any condition for $N$; see \cite{CM} for details.
\end{remark}

\section{Proof of Theorem \ref{thm:InvGM}} \label{sec:ProofThm2}

Let $d, \beta, \sigma$ be as in Introduction.
We also use symbols defined in Section \ref{sec:Wg}.
Our starting point for the proof of Theorem \ref{thm:InvGM} is 
the following lemma.

\begin{lem}  \label{lem:InvGM0}
Let $W \sim W_d(\beta,\sigma;\bR)$ and let $s_1,\dots,s_n \in \Sym(d)$.
Put $\gamma=\beta-\frac{d+1}{2}$ and suppose $\gamma>0$.
Then
$$
\tr(\sigma^{-1} s_1) \tr (\sigma^{-1} s_2) \cdots \tr (\sigma^{-1} s_n) = 
(-1)^n \sum_{\pi \in S_n} (-\gamma)^{\nu(\pi)}
\bE[ R_{\pi}(W^{-1};s_1,\dots,s_n)],
$$
where $R_\pi(\cdot;\cdots)$ is defined in Section \ref{sec:ProofThm1}.
\end{lem}

\begin{proof}
We can obtain the proof in the same way to  \cite[Theorem 3]{GLM1}.
Therefore we omit it here.
(The assumption $\gamma= \beta -\frac{d+1}{2}>0$ implies that  
the real Wishart distribution $\mf{W}_{d,\beta,\sigma}$ has 
the density $f(w;d,\beta,\sigma)$ given by \eqref{eq:Density},
and that $f(w;d,\beta,\sigma)$ vanishes on the boundary of $\Omega$.
Therefore, we can apply Stokes' formula for $f$; see pages 298--299 in \cite{GLM1}.) 
\end{proof}

\begin{lem} \label{lem:InvGM1}
Let $W$ and $\gamma$ be as in Lemma \ref{lem:InvGM0}.
Given indices $k_1,k_2,\dots,k_{2n}$ 
from
$\{1,\dots,d\}$, we have
\begin{equation} \label{eq:InvGM1}
\sigma^{k_1 k_2} \sigma^{k_3 k_4}\cdots 
\sigma^{k_{2n-1} k_{2n}}=
(-1)^n 2^{-n}\sum_{\mf{m} \in \mcal{M}(2n)}
(-2\gamma)^{\kappa(\mf{m})} \bE
\left[
\prod_{ \{p,q\} \in \mf{m} }
W^{k_p k_q}
\right]. 
\end{equation}
\end{lem}

\begin{proof}
By using Lemma \ref{lem:InvGM0},
one can prove it in 
the same way to the proof of Theorem \ref{thm:GM}.
Indeed, 
applying Lemma \ref{lem:InvGM0} with
$s_j=(E_{k_{2j-1},k_{2j}}+E_{k_{2j},k_{2j-1}})/2$ \ $(1 \le j \le n)$,
and using \eqref{eq:R_P} and Proposition \ref{prop:AlphaHf},
we see that
\begin{align*}
& \sigma^{k_1 k_2} \sigma^{k_3 k_4} \cdots \sigma^{k_{2n-1} k_{2n}} \\
=& (-1)^n 2^{-n} \sum_{ \pi \in S_n} (-\gamma)^{\nu(\pi)} 
\bE[R_{\pi}(W^{-1}; E_{k_1 k_2}+E_{k_2 k_1}, \dots, E_{k_{2n-1}k_{2n}} +E_{k_{2n}k_{2n-1}})] \\
=& (-1)^n 2^{-n} \sum_{ \pi \in S_n} (-\gamma)^{\nu(\pi)}
\bE \left[P_\pi \( (W^{k_p k_q})_{p,q \in [2n]} \) \right]\\
=& (-1)^n 2^{-n} \bE\left[ \hf_{-2\gamma} (W^{k_p k_q})_{p,q \in [2n]} \right].
\end{align*} 
\end{proof}

Suppose $\gamma >n-1$. Then $\Wg^O(g;-2\gamma)$ ($g \in S_{2n}$) can be defined (see
Subsection \ref{subsec:Wg}).
Set
\begin{equation} \label{eq:tWg}
\tWg(g;\gamma)= (-1)^n 2^n \Wg^O(g;-2\gamma) = \frac{2^n n!}{(2n)!} 
(-1)^n 2^n \sum_{\lambda \vdash n} \frac{f^{2\lambda}}{C_\lambda(-2\gamma)} \omega^\lambda(g)
\qquad (g \in S_{2n}).
\end{equation}

We finally prove Theorem \ref{thm:InvGM}.
Recall that the functions $g \mapsto \kappa(g)$ and 
$g \mapsto \Wg(g;z)$ are $H_n$-biinvariant.
We can rewrite \eqref{eq:InvGM1} in the form
$$
\sigma^{k_1 k_2} \sigma^{k_3 k_4}\cdots 
\sigma^{k_{2n-1} k_{2n}}=
(-1)^n 2^{-n} (2^n n!)^{-1}\sum_{g \in S_{2n}}
(-2\gamma)^{\kappa(g)} \bE
\left[
W^{k_{g(1)} k_{g(2)}} \cdots W^{k_{g(2n-1)} k_{g(2n)}}
\right]
$$
by the coset decomposition \eqref{eq:LeftDecom}.
Therefore, the right hand side of \eqref{eq:InvGM2} is equal to
\begin{align*}
&(-1)^n 2^{n} (2^n n!)^{-1} \sum_{g' \in S_{2n}}
\mr{Wg}^O(g'; -2\gamma)
\sigma^{k_{g'(1)} k_{g'(2)}} \cdots \sigma^{k_{g'(2n-1)} k_{g'(2n)}} \\
=& 
(2^n n!)^{-2} \sum_{g, g' \in S_{2n}}
(-2\gamma)^{\kappa(g)}\mr{Wg}^O(g'; -2\gamma) 
\bE
\left[
W^{k_{g'g(1)} k_{g'g(2)}} \cdots W^{k_{g'g(2n-1)} k_{g'g(2n)}}
\right] \\
=& (2^n n!)^{-2} \sum_{g, g'' \in S_{2n}}
(-2\gamma)^{\kappa(g)}\mr{Wg}^O(g'' g^{-1} ; -2\gamma) 
\bE
\left[
W^{k_{g''(1)} k_{g''(2)}} \cdots W^{k_{g''(2n-1)} k_{g''(2n)}}
\right]
\end{align*}
by letting $g''=g'g$.
Since Lemma \ref{lem:Gwg} implies
$$
\sum_{g \in S_{2n}}
z^{\kappa(g)}\mr{Wg}^O(g'' g^{-1} ; z) =
\begin{cases}
2^n n! & \text{if $g'' \in H_n$,} \\
0 & \text{otherwise},
\end{cases}
$$
the last equation equals
$$
(2^n n!)^{-1} \sum_{g'' \in H_n}
\bE
\left[
W^{k_{g''(1)} k_{g''(2)}} \cdots W^{k_{g''(2n-1)} k_{g''(2n)}}
\right]
= \bE[W^{k_1 k_2} W^{k_3 k_4}\cdots 
W^{k_{2n-1} k_{2n}}].
$$
Hence we have proved Theorem \ref{thm:InvGM}.

\begin{remark} \label{remark:gamma}
Theorem \ref{thm:InvGM} holds true for 
any positive real number $\gamma$ such that
$C_\lambda(-2\gamma) \not=0$ for all $\lambda \vdash n$.
\end{remark}

\begin{remark}
The complex-Wishart version of Theorem \ref{thm:InvGM} is obtained 
by Graczyk et al. \cite{GLM1}.
They employ a class function on $S_n$ defined by
$$
\mr{Wg}^U(\pi; -q)= \frac{1}{n!} \sum_{\lambda \vdash n}
\frac{f^\lambda}{\prod_{(i,j)\in \lambda} (-q+j-i)} \chi^\lambda(\pi)
\qquad (\pi \in S_n),
$$
where $q >n-1$ is a parameter in \cite{GLM1}, corresponding to our $\gamma$. 
The function $\mr{Wg}^U(\pi; N)$ coincides with  the Weingarten fucntion
for the unitary group $U(N)$, studied in \cite{C} (see also \cite{MN}).
\end{remark}

\section{Applications} \label{sec:Appl}

In this section, we give applications of Theorem \ref{thm:GM} and
Theorem \ref{thm:InvGM}.

\subsection{Mixed moments of traces}

Recall the symbol $R_\pi(x;m_1,\dots,m_n)$ defined in Section \ref{sec:ProofThm1},
where $x$ is a $d \times d$ symmetric matrix, 
$m_1,\dots,m_n$ are $d \times d $ complex matrices, and $\pi \in S_n$.
For example,
\begin{align*}
R_{(1 \to 3 \to 2 \to 4 \to 1)}(x;m_1,m_2,m_3,m_4)=&
\tr ( x m_1 x m_3 x m_2 x m_4), \\
R_{(1 \to 4 \to 5 \to 1)(2 \to 7 \to 2)(6 \to 6)}
(x; m_1,m_2,\dots,m_7) =& \tr (x m_1 x m_4 x m_5) \tr(x m_2 x m_7 ) \tr (x m_6). 
\end{align*}
Thus $R_\pi(x;m_1,\dots,m_n)$ is a product of traces of the form
$\tr( x  m_{i_1} x m_{i_2} \cdots  x m_{i_k})$.
Our purpose in this section is to compute moments of the forms
$$
\bE[ R_\pi(W;m_1,\dots,m_n)] \qquad \text{and} \qquad 
\bE[ R_\pi(W^{-1};m_1,\dots,m_n)]
$$
where $W \sim W_d(\beta,\sigma;\bR)$ as usual.

First we observe a simple example.

\begin{example} \label{ex:mixMoment}
We compute $\bE[\tr(W m_1 W m_2)]$.
Expanding the trace, we have
$$
\bE[\tr(W m_1 W m_2)] =\sum_{k_1,k_2,k_3,k_4} (m_1)_{k_2 k_3} (m_2)_{k_4 k_1}  
\bE [W_{k_1 k_2} W_{k_3 k_4}].
$$
From Theorem \ref{thm:GM} or  \eqref{eq:2degreeW}, it is equal to 
\begin{align*}
&\sum_{k_1,k_2,k_3,k_4} (m_1)_{k_2 k_3} (m_2)_{k_4 k_1} \(
\beta^2 \sigma_{k_1 k_2} \sigma_{k_3 k_4} + \frac{\beta}{2} 
\sigma_{k_1 k_3} \sigma_{k_2 k_4} + \frac{\beta}{2}
\sigma_{k_1 k_4} \sigma_{k_2 k_3} \) \\
=& \beta^2 \tr (\sigma m_1 \sigma m_2) + \frac{\beta}{2} \tr (\sigma m_1^t \sigma m_2)
+ \frac{\beta}{2} \tr(\sigma m_1) \tr (\sigma m_2), 
\end{align*}
where $m^t$ is the transpose of $m$. In other words,
$$
\bE[R_{(1 \to 2 \to 1)}(W;m_1,m_2)] = \beta^2 R_{(1 \to 2 \to 1)}(\sigma;m_1,m_2)
+\frac{\beta}{2} R_{(1 \to 2 \to 1)}(\sigma;m_1^t,m_2) + \frac{\beta}{2}
R_{(1 \to 1)(2 \to 2)} (\sigma;m_1,m_2).
$$
This example indicates that we should deal with not only $m_1,\dots,m_n$
but also with their transposes $m_1^t,\dots,m^t_n$.
\end{example}

Given a matrix $m=(m_{ij})$ and a signature $\epsilon \in \{-1,+1 \}$, we put
$$
m^{\epsilon}= \begin{cases}
m & \text{if $\epsilon=+1$,} \\
m^t & \text{if $\epsilon=-1$}.
\end{cases}
$$

Let $m_1,\dots,m_n$ be $d \times d$ complex matrices and 
let $x=(x_{i,j})$ be a $d \times d$ real symmetric matrix.
Given a permutation $g \in S_{2n}$, we define $T_g(x;m_1,\dots,m_n)$ by
$$
T_g(x;m_1,\dots,m_n) 
=\sum_{j_1,\dots,j_{2n} =1}^d
(m_1)_{j_1, j_2} (m_2)_{j_3, j_4} \cdots (m_n)_{j_{2n-1}, j_{2n}}
x_{j_{g(1)}, j_{g(2)}} x_{j_{g(3)}, j_{g(4)}} \cdots x_{j_{g(2n-1)},j_{g(2n)}}.
$$
In our situation, the symbol $T_g$ is more useful than $R_\pi$.

Given $\pi \in S_n$, we denote by $\tilde{\pi}$ the permutation in $S_{2n}$ given by
$\tilde{\pi}(2j-1)=2\pi(j)-1$ and $\tilde{\pi}(2j)=2 j$ for $j=1,2,\dots,n$.
Denote by $\zeta_i$ the transposition $(2i -1 \to 2i \to 2i-1)$.

\begin{lem} \label{lem:RT}
For $\pi \in S_n$ and $\epsilon_1,\dots, \epsilon_n \in \{\pm 1\}$ we have
$$
R_{\pi}(x;m_1^{\epsilon_1},\dots, m_n^{\epsilon_n}) = T_g(x;m_1,\dots, m_n)
\qquad \text{with $g= \( \prod_{ i: \epsilon_i =-1 } \zeta_i \) \cdot \tilde{\pi}$}.
$$
\end{lem}

\begin{proof}
First we will show
\begin{equation} \label{eq:RT1}
R_{\pi}(x;m_1,\dots, m_n) = T_{\tilde{\pi}}(x;m_1,\dots, m_n).
\end{equation}
Take a cycle $c=(c_1 \to c_2 \to \cdots \to c_r \to c_1)$  in $\pi$.
Then we see that
\begin{align*}
& \sum_{j_{2c_1-1},j_{2c_1},\dots, j_{2c_r-1},j_{2 c_r} } \prod_{k=1}^r
(m_{c_k})_{j_{2c_k-1}, j_{2c_k}} x_{j_{\tilde{\pi}(2c_k-1)},j_{\tilde{\pi}(2c_k)}} \\
=& \sum_{j_{2c_1-1},j_{2c_1},\dots, j_{2c_r-1},j_{2 c_r} }\prod_{k=1}^r
(m_{c_k})_{j_{2c_k-1}, j_{2c_k} } x_{j_{2\pi(c_k)-1},j_{2c_k}} \\
=& \sum_{j_{2c_1-1},j_{2c_1},\dots, j_{2c_r-1},j_{2 c_r} } 
(m_{c_1})_{j_{2c_1-1},j_{2c_1}} x_{j_{2c_1}, j_{2c_2-1}} 
(m_{c_2})_{j_{2c_2-1},j_{2c_2}} x_{j_{2c_2}, j_{2c_3-1}} \cdots 
(m_{c_r})_{j_{2c_r-1},j_{2c_r}} x_{j_{2c_r}, j_{2c_1-1}} \\
=& \tr(m_{c_1} x m_{c_2} x \cdots m_{c_r} x) = R_c(x;m_1,\dots,m_n).
\end{align*}
We obtain \eqref{eq:RT1} by taking the product over all cycles in $\pi$.

Next we will show
\begin{equation} \label{eq:RT2}
R_{\pi}(x;m_1,\dots,m_i^{t},\dots, m_n) = T_{\zeta_i\tilde{\pi}} (x;m_1,\dots, m_n).
\end{equation}
We have 
$$
T_{\zeta_i \tilde{\pi}} (x;m_1,\dots, m_n) =
\sum_{j_1,\dots,j_{2n} } \prod_{k=1}^n  (m_k)_{j_{2k-1}, j_{2k}}
x_{j_{\zeta_i \tilde{\pi} (2k-1)},j_{\zeta_i \tilde{\pi}(2n)}}. 
$$
Letting $j_k'= j_{\zeta_i(k)}$ for all $k=1,2,\dots,2n$,
it is equal to
\begin{align*}
&  \sum_{j_1',\dots,j_{2n}'} \prod_{k=1}^n  (m_k)_{j'_{\zeta_i(2k-1)}, j'_{\zeta_i(2k)}}
x_{j'_{ \tilde{\pi} (2k-1)},j'_{\tilde{\pi}(2k)}} 
\\
=& \sum_{j'_1,\dots,j'_{2n}} (m_i^t)_{j'_{2i-1},j'_{2i}} 
x_{j'_{ \tilde{\pi} (2i-1)},j'_{\tilde{\pi}(2i)}}
\prod_{k \not= i}  (m_k)_{j'_{2k-1}, j_{2k}'}
x_{j'_{ \tilde{\pi} (2k-1)},j'_{\tilde{\pi}(2k)} } \\
=& T_{\tilde{\pi}}(x;m_1,\dots,m_i^t,\dots,m_n).
\end{align*}
Therefore, \eqref{eq:RT2} follows by \eqref{eq:RT1}.
Now the result can be obtained from \eqref{eq:RT1} and \eqref{eq:RT2}.
\end{proof}

\begin{example}
Consider
$$
\tr (x m_1 x m_4^t x m_5^t x m_2) \tr (x m_3 x m_7^t) \tr(x m_6),
$$
which is equal to $R_{\pi}(x;m_1^{\epsilon_1},\dots,m_7^{\epsilon_7})$ with
\begin{align*}
\pi=& (1 \to 4 \to 5 \to 2 \to 1)(3 \to 7 \to 3)(6 \to 6) \in S_7,\\
(\epsilon_1,\dots,\epsilon_7)=&(+1,+1,+1,-1,-1,+1,-1).
\end{align*}
It coincides with $T_{g}(x;m_1,\dots,m_7)$, 
where $g=\zeta_{4} \zeta_{5} \zeta_{7} \tilde{\pi}$, i.e.,
$$
g= (7 \to 8 \to 7)(9 \to 10 \to 9)(13 \to 14 \to 13)
(1 \to 7 \to 9 \to 3 \to 1)(5 \to 13 \to 5)(11 \to 11). 
$$
\end{example}

\begin{lem} \label{lem:Tright}
The function $S_{2n} \ni g \mapsto T_g(x;m_1,\dots,m_n)$ is 
right $H_n$-invariant:
$$
T_{g \zeta} (x;m_1,\dots,m_n)= T_{g}  (x;m_1,\dots,m_n) \qquad
\text{for all $\zeta \in H_n$ and $g \in S_{2n}$.}
$$ 
\end{lem}

\begin{proof}
It is enough to check for $\zeta=(2i-1 \to 2i \to 2i-1)$ and $(2i-1 \to 2j-1 \to 2i-1)
(2i \to 2j \to 2i)$
because $H_n$ is generated by them. However, it is clear.
\end{proof}

The moment of the form $\bE[ R_\pi(W^{\pm 1}; m_1^{\epsilon_1}, \dots,m_n^{\epsilon_n})]$
may be given by $\bE[ T_{g}  (W^{\pm 1};m_1,\dots,m_n)]$ with some $g \in S_{2n}$. 
Hence we now compute the moments $\bE[T_g(W^{\pm 1} ;m_1,\dots,m_n)]$.
First of all, we note that the formulas in Theorem \ref{thm:GM} and 
Theorem \ref{thm:InvGM} can be expressed  in the forms
\begin{align}
\bE[ W_{k_1 k_2}  \cdots W_{k_{2n-1}, k_{2n}}]=& 2^{-n} (2^n n!)^{-1}
\sum_{g \in S_{2n}} (2\beta)^{\kappa(g)} \sigma_{k_{g(1)},k_{g(2)}} \cdots 
\sigma_{k_{g(2n-1)},k_{g(2n)}}, \label{eq:thm1kai} \\
\bE[ W^{k_1 k_2}  \cdots W^{k_{2n-1}, k_{2n}}]=&  (2^n n!)^{-1}
\sum_{g \in S_{2n}} \tWg(g;\gamma) \sigma^{k_{g(1)},k_{g(2)}} \cdots 
\sigma^{k_{g(2n-1)},k_{g(2n)}}.
\end{align}

\begin{thm} \label{thm:Trace1}
Let $W \sim W_{d}(\beta,\sigma;\bR)$ and let $\gamma$ be as in Theorem \ref{thm:InvGM}.
Let $m_1,\dots,m_n$ be $d \times d$ matrices and let $ g \in S_{2n}$. Then
\begin{align*}
\bE [T_g(W;m_1,\dots,m_n)]=&
2^{-n} \sum_{\mf{n} \in \mcal{M}(2n)} 
(2\beta)^{\kappa(g^{-1} \mf{n})} 
T_{\mf{n}} (\sigma;m_1,\dots,m_n), \\
\bE [T_g(W^{-1};m_1,\dots,m_n)]=&
\sum_{\mf{n} \in \mcal{M}(2n)} 
\widetilde{\mr{Wg}}(g^{-1} \mf{n};\gamma)
T_{\mf{n}} (\sigma^{-1};m_1,\dots,m_n).
\end{align*}
\end{thm}

\begin{proof}
Using \eqref{eq:thm1kai} (or Theorem \ref{thm:GM}),
\begin{align*}
&  \bE [T_g(W;m_1,\dots,m_n)] \\
=& \sum_{j_1,\dots,j_{2n}} \(\prod_{k=1}^n (m_k)_{j_{2k-1},j_{2k}}\)  
\bE [W_{j_{g(1)}, j_{g(2)}} \cdots W_{j_{g(2n-1)}, j_{g(2n)}}] \\
=& \sum_{j_1,\dots,j_{2n}} \(\prod_{k=1}^n (m_k)_{j_{2k-1},j_{2k}} \) 
2^{-n} (2^n n!)^{-1} \sum_{g' \in S_{2n}} 
(2\beta)^{\kappa(g')} \sigma_{j_{g g'(1)}, j_{g g'(2)}} \cdots 
\sigma_{j_{g g'(2n-1)}, j_{g g'(2n)}} \\
\intertext{and, letting $h=g g'$,}
=& 2^{-n} (2^n n!)^{-1} \sum_{h \in S_{2n}} 
(2\beta)^{\kappa(g^{-1} h)} 
\sum_{j_1,\dots,j_{2n}} \prod_{k=1}^n (m_k)_{j_{2k-1},j_{2k}}
\sigma_{j_{h(2k-1)}, j_{h(2k)}} \\
=& 2^{-n} (2^n n!)^{-1} \sum_{h \in S_{2n}} 
(2\beta)^{\kappa(g^{-1} h)} 
T_h (\sigma;m_1,\dots,m_n) \\
=& 2^{-n} \sum_{\mf{n} \in \mcal{M}(2n)} 
(2\beta)^{\kappa(g^{-1} \mf{n})} 
T_{\mf{n}} (\sigma;m_1,\dots,m_n).
\end{align*}
Here the last equality follows from Lemma \ref{lem:Tright} and \eqref{eq:LeftDecom}.
Thus the first formula has been proved.
The same applies to the second formula.
\end{proof}

It follows from Lemma \ref{lem:RT} and Theorem \ref{thm:Trace1} that,
for $\pi \in S_n$ and $(\epsilon_1,\dots, \epsilon_n) \in \{-1,+1\}^n$,
\begin{align}
\bE [R_\pi(W;m_1^{\epsilon_1},\dots,m_n^{\epsilon_n})]=&
2^{-n} \sum_{\mf{n} \in \mcal{M}(2n)} 
(2\beta)^{\kappa(g^{-1} \mf{n})} 
T_{\mf{n}} (\sigma;m_1,\dots,m_n), \label{eq:MixTraceMomentW} \\
\bE [R_\pi(W;m_1^{\epsilon_1},\dots,m_n^{\epsilon_n})]=&
\sum_{\mf{n} \in \mcal{M}(2n)} 
\widetilde{\mr{Wg}}(g^{-1} \mf{n};\gamma)
T_{\mf{n}} (\sigma^{-1};m_1,\dots,m_n),
\end{align}
where $g$ is as in Lemma \ref{lem:RT}.
We remark that \eqref{eq:MixTraceMomentW} is equivalent to 
\cite[Corollary 14]{GLM2}.

\subsection{Averages of invariant polynomials}

Given a partition $\lambda$ of $n$,
we define two functions $\bm{Z}_\lambda$ and $\bm{p}_\lambda$ on 
$\Omega=\Sym^+(d)$ by
$$
\bm{Z}_\lambda(x)= Z_\lambda(a_1,a_2,\dots,a_d,0,0,\dots) \qquad
\text{and} \qquad
\bm{p}_\lambda(x)= p_\lambda(a_1,a_2,\dots,a_d,0,0,\dots),
$$
where $a_1,\dots,a_d$ are eigenvalues of $x \in \Omega$,
and $Z_\lambda, p_\lambda$ are symmetric functions defined in Subsection \ref{subsec:zonal}.
In particular, we have
$$
\bm{p}_\lambda(x)= \prod_{i=1}^{\ell(\lambda)} \tr (x^{\lambda_i})=
\prod_{r \ge 1} \( \tr(x^r) \)^{m_r(\lambda)},
$$
where $m_r(\lambda)$ is the multiplicity of $r$ in $\lambda$.
From \eqref{eq:Zp} and \eqref{eq:pZ}, we have
\begin{equation} \label{eq:ZpZ}
\bm{Z}_\lambda= 
2^n n! \sum_{\rho \vdash n} 2^{-\ell(\rho)} z_\rho^{-1} \omega^\lambda_\rho \bm{p}_\rho 
\qquad \text{and} \qquad
\bm{p}_\rho= \frac{2^n n!}{(2n)!} \sum_{\lambda \vdash n} f^{2\lambda} \omega^\lambda_\rho
\bm{Z}_\lambda.
\end{equation}

Recall  $C_\lambda(z)=\prod_{(i,j) \in \lambda} (z+2j-i-1)$.
The following theorem, derived from Theorem \ref{thm:GM} and Theorem \ref{thm:InvGM},
is exactly the real case of Proposition 5 and 6 in \cite{LM1}.

\begin{thm} \label{thm:Invariants}
Let $W \sim W_{d}(\beta,\sigma;\bR)$ and let $\gamma$ be as in Theorem \ref{thm:InvGM}.
For a partition $\lambda$ of $n$,
\begin{align*}
\bE[ \bm{Z}_\lambda(W)] =& 2^{-n} C_\lambda(2\beta) \bm{Z}_\lambda(\sigma). \\
\bE[ \bm{Z}_\lambda(W^{-1})] =& (-1)^n 2^{n} C_\lambda(-2\gamma)^{-1} \bm{Z}_\lambda(\sigma^{-1}).
\end{align*}
\end{thm}

\begin{proof}
First of all, we note that 
$$
\bm{p}_\rho(x)= T_{g}(x;\underbrace{I_d,\dots,I_d}_n)
$$
for a permutation $g$ in $S_{2n}$ of coset-type $\rho$ 
and for a matrix $x$ in $\Omega$.
Indeed, since the function $S_{2n} \ni g \mapsto T_{g}(x;I_d,\dots,I_d)$
is $H_n$-biinvariant, the image depends only on the coset-type.
If $\pi$ is a permutation in $S_{n}$ of cycle-type $\rho$, then
$\tilde{\pi}$ is of coset-type $\rho$, and therefore
$T_{g}(x;I_d,\dots,I_d)= T_{\tilde{\pi}}(x;I_d,\dots,I_d)
= R_\pi(x;I_d,\dots,I_d)=
\bm{p}_\rho(x)$ by Lemma \ref{lem:RT}.

From the first formula in \eqref{eq:ZpZ} and 
 the double decomposition \eqref{eq:doubledecomposition}, we have 
\begin{align*}
\bE[ \bm{Z}_\lambda(W)]
=& 2^n n! \sum_{\rho \vdash n} 2^{-\ell(\rho)} z_\rho^{-1} \omega^\lambda_\rho \bE[\bm{p}_\rho(W)] \\
=& 2^n n! \sum_{\rho \vdash n} 2^{-\ell(\rho)} z_\rho^{-1}  \frac{1}{|H_\rho|}
\sum_{g \in H_\rho}\omega^\lambda(g) \bE[T_g(W;I_d,\dots,I_d)] \\
=& (2^n n!)^{-1} \sum_{g \in S_{2n}}\omega^\lambda(g) 
\bE[T_g(W;I_d,\dots,I_d)].
\end{align*}
It follows from Theorem \ref{thm:Trace1} that
\begin{align*}
\bE[ \bm{Z}_\lambda(W)]
=& (2^n n!)^{-2} \sum_{g \in S_{2n}}\omega^\lambda(g)  2^{-n} 
\sum_{g' \in S_{2n}} (2\beta)^{\kappa(g^{-1} g')} 
T_{g'}(\sigma;I_d,\dots,I_d) \\
=& (2^n n!)^{-2}  2^{-n}  \sum_{g' \in S_{2n}} \((\omega^\lambda *G^O(\cdot;2\beta)) (g') \)
T_{g'}(\sigma;I_d,\dots,I_d).
\end{align*}
Since $\omega^\lambda *G^O(\cdot;z)
=2^n n! C_\lambda(z) \omega^\lambda$
by \eqref{eq:ExpandG} and \eqref{eq:zonalorthogonal}, we have
$$
\bE[ \bm{Z}_\lambda(W)] = (2^n n!)^{-1}
2^{-n} C_\lambda(2\beta) \sum_{g' \in S_{2n}} \omega^\lambda(g')   T_{g'}(\sigma;I_d,\dots,I_d).
$$
Since 
$$
\sum_{g' \in S_{2n}} \omega^\lambda(g')   T_{g'}(\sigma;I_d,\dots,I_d)
= \sum_{\rho \vdash n} |H_\rho| \omega^\lambda_\rho \bm{p}_\rho(\sigma)
=\sum_{\rho \vdash n} \frac{(2^n n!)^2}{2^{\ell(\rho)} z_\rho} 
\omega^\lambda_\rho \bm{p}_\rho(\sigma) = 2^n n! \bm{Z}_\lambda(\sigma)
$$
by the first formula in \eqref{eq:ZpZ},
our first result follows.
The proof of our second result is similar.
\end{proof}

The following is equivalent to the real case of \cite[Theorem 2]{LM1}.

\begin{cor} \label{cor:Trace}
Let $W \sim W_{d}(\beta,\sigma;\bR)$ and let $\gamma$ be as in Theorem \ref{thm:InvGM}.
For a partition $\mu$ of $n$,
\begin{align*}
\bE[ \bm{p}_\mu(W)] =& \frac{(2^n n!)^2}{(2n)!}  
\sum_{\rho \vdash n } 2^{-\ell(\rho)}z_\rho^{-1} \( 2^{-n}\sum_{\lambda \vdash n}  C_\lambda(2\beta)
f^{2\lambda} \omega^\lambda_\mu \omega^\lambda_\rho \) \bm{p}_\rho(\sigma), \\
\bE[ \bm{p}_\mu(W^{-1})] =& \frac{(2^n n!)^2}{(2n)!}  
\sum_{\rho \vdash n } 2^{-\ell(\rho)}z_\rho^{-1} \( (-1)^n 2^{n}\sum_{\lambda \vdash n}  C_\lambda(-2\gamma)^{-1}
f^{2\lambda} \omega^\lambda_\mu \omega^\lambda_\rho \) \bm{p}_\rho(\sigma^{-1}).
\end{align*}
\end{cor}

\begin{proof}
They follow from Theorem \ref{thm:Invariants} and  \eqref{eq:ZpZ}.
\end{proof}

\begin{cor} \label{cor:Trace2}
Let $W \sim W_{d}(\beta,\sigma;\bR)$ and let $\gamma$ be as in Theorem \ref{thm:InvGM}.
Then
\begin{align*}
\bE[(\tr W)^n] =& \sum_{\rho \vdash n} \frac{n!}{z_\rho} \beta^{\ell(\rho)} 
\bm{p}_\rho(\sigma), \\
\bE[(\tr W^{-1} )^n] =& 
\sum_{\rho \vdash n}2^{n-\ell(\rho)} \frac{n!}{z_\rho} 
\tWg(\rho;\gamma) \bm{p}_\rho(\sigma^{-1}).
\end{align*}
\end{cor}

\begin{proof}
The first result follows by
letting $\mu=(1^n)$ in Corollary \ref{cor:Trace} and by using 
the second formula in \eqref{eq:specialZp}.
The second one also follows by \eqref{eq:DefWg}.
\end{proof}

\subsection{Examples for low degrees}

If we apply our theorems 
(Theorem \ref{thm:GM} -- Corollary \ref{cor:Trace2}),
we can compute various Wishart moments for low degrees 
$n=1,2,3,4$ easily; see Appendix.

\section*{Acknowledgements}

I would like to thank Piotr Graczyk for getting me interested 
in  Wishart distributions on May 2009,
and thank Hideyuki Ishi, 
who organized the meeting where I met P. Graczyk.
I also thank Yasuhide Numata for his talk on noncentral Wishart distributions
in March 2010.



\noindent
\textsc{Sho Matsumoto \\
Graduate School of Mathematics, Nagoya University, Nagoya, 464-8602, Japan.} \\
E-mail: \verb|sho-matsumoto@math.nagoya-u.ac.jp|

\newpage
\appendix
\section{Appendix: Examples for low degrees} \label{sec:Example}

We give explicit examples of our theorems obtained in the present papers.
Let $W \sim W_d(\beta,\sigma;\bR)$ and set $\gamma=\beta-\frac{d+1}{2}$ as usual.
Let $m_1,m_2,\dots$ be $d \times d$ matrices.

\subsection*{Degree 1}

Suppose $\gamma >0$.
It follows from Theorem \ref{thm:GM} and Theorem \ref{thm:InvGM} that
\begin{equation}
\bE[W_{ij}]= \beta \sigma_{ij} \qquad  \text{and} \qquad
\bE[W^{ij}]= \frac{1}{\gamma} \sigma^{ij}
\end{equation}
for $1 \le i,j \le d$.
It is immediate to see that
\begin{align}
\bE[W]=& \beta \sigma, & \bE[W^{-1}]=& \gamma^{-1} \sigma^{-1}, \\
\bE[\tr (W m_1)] =& \beta \tr (\sigma m_1), & 
\bE[\tr(W^{-1}m_1)]=& \gamma^{-1} \tr(\sigma^{-1}m_1).
\end{align}

\subsection*{Degree 2}

Suppose $\gamma >0$ but $\gamma \not=1$ (see Remark \ref{remark:gamma}).
From \eqref{eq:tWg} and Example \ref{ex:Wgz},
\begin{align*}
\widetilde{\mr{Wg}}(\{ \{1,2\}, \{3,4\} \};\gamma)=& 
\frac{2\gamma-1}{\gamma(\gamma-1)(2\gamma+1)}, \\
\widetilde{\mr{Wg}}(\{ \{1,3\}, \{2,4\} \};\gamma)= 
\widetilde{\mr{Wg}}(\{ \{1,4\}, \{2,3\} \};\gamma)=& 
\frac{1}{\gamma(\gamma-1)(2\gamma+1)}.
\end{align*}

It follows from Theorem \ref{thm:GM} and Theorem \ref{thm:InvGM}  that
\begin{align}
\bE[W_{k_1 k_2} W_{k_3 k_4}] =&
\beta^2 \sigma_{k_1 k_2} \sigma_{k_3 k_4} 
+\frac{\beta}{2} (\sigma_{k_1 k_3} \sigma_{k_2 k_4} +  \sigma_{k_1 k_4} \sigma_{k_2 k_3}), \\
\bE[ W^{k_1 k_2} W^{k_3 k_4}] =& \frac{1}{\gamma(\gamma-1)(2\gamma+1)} 
\Big[ (2\gamma-1) \sigma^{k_1 k_2} \sigma^{k_3 k_4} +
\sigma^{k_1 k_3} \sigma^{k_2 k_4}+\sigma^{k_1 k_4} \sigma^{k_2 k_3} \Big],
\end{align}
for $(k_1,k_2,k_3,k_4) \in [d]^4$.

The average for the $(i,j)$-entry of $W^{2}$ is
\begin{align*}
\bE \left[ \sum_{k=1}^d W_{ik} W_{kj} \right] 
=& \beta^2 \sum_{k=1}^d \sigma_{ik} \sigma_{kj}
+ 
 \frac{\beta}{2} \sum_{k=1}^d (\sigma_{ik} \sigma_{kj} +\sigma_{ij} \sigma_{kk}) \\
=& \(\beta^2+\frac{\beta}{2} \) (\sigma^2)_{ij} + \frac{\beta}{2} (\tr \sigma) \sigma_{ij},
\end{align*}
and the average for the $(i,j)$-entry of $W^{-2}$ is
\begin{align*}
\bE \left[ \sum_{k=1}^d W^{ik} W^{kj} \right] 
=& \frac{1}{\gamma(\gamma-1)(2\gamma+1)} \Big[ (2\gamma-1) \sum_{k=1}^d \sigma^{ik} \sigma^{kj}
+ 
\sum_{k=1}^d (\sigma^{ik} \sigma^{kj} +\sigma^{ij} \sigma^{kk})\Big] \\
=& \frac{1}{\gamma(\gamma-1)(2\gamma+1)}
\( 2\gamma (\sigma^{-2})_{ij} + \tr(\sigma^{-1}) \sigma^{ij} \). 
\end{align*}
Therefore
\begin{align}
\bE [W^2]=& \(\beta^2+\frac{\beta}{2} \) \sigma^2 + \frac{\beta}{2} (\tr\sigma) \sigma, \\
\bE [W^{-2}] =&
\frac{1}{\gamma(\gamma-1)(2\gamma+1)}
\( 2\gamma \sigma^{-2} + \tr(\sigma^{-1}) \sigma \).
\end{align}

As we saw in  Example \ref{ex:mixMoment}, 
\begin{equation}
\bE[ \tr (W m_1 W m_2)]= \beta^2 \tr(\sigma m_1 \sigma m_2) +\frac{\beta}{2}
\tr(\sigma m_1^t \sigma m_2) + \frac{\beta}{2}\tr(\sigma m_1) \tr(\sigma m_2),
\end{equation}
and in a similar way we have
\begin{align}
\bE[ \tr (W^{-1} m_1 W^{-1} m_2)]=& \frac{1}{\gamma(\gamma-1)(2\gamma+1)} 
\Big[ (2\gamma-1) \tr(\sigma^{-1} m_1 \sigma^{-1} m_2)  \\
& \qquad +\tr(\sigma^{-1} m_1^t \sigma^{-1} m_2) +\tr(\sigma^{-1} m_1) \tr(\sigma^{-1} m_2) \Big]
. \notag
\end{align}
Moreover 
\begin{align}
\bE[\tr(Wm_1) \tr(Wm_2)]=&
\beta^2 \tr(\sigma m_1) \tr(\sigma m_2) +\frac{\beta}{2} \tr(\sigma m_1 \sigma m_2)
+\frac{\beta}{2} \tr(\sigma m_1^t \sigma m_2), \\
\bE[\tr(W^{-1}m_1) \tr(W^{-1}m_2)]=&
\frac{1}{\gamma(\gamma-1)(2\gamma+1)} 
\Big[ (2\gamma-1) \tr(\sigma^{-1} m_1) \tr( \sigma^{-1} m_2)  \\
& \qquad +\tr(\sigma^{-1} m_1 \sigma^{-1} m_2) +\tr(\sigma^{-1} m_1^t \sigma^{-1} m_2) \Big]. \notag
\end{align}

\subsection*{Degree 3}

Suppose $\gamma>0$ but $\gamma \not=1,2$.
From \eqref{eq:tWg} and a list in \cite{CM} (see also \cite{CS}),
the $\tWg(\rho;\gamma)$ \ $( \rho \vdash 3)$ are given by
\begin{align*}
\widetilde{\mr{Wg}}((3);\gamma)=& \frac{1}{u_3(\gamma)}, &
\widetilde{\mr{Wg}}((2,1);\gamma)=& \frac{\gamma-1}
{u_3(\gamma)}, &
\widetilde{\mr{Wg}}((1^3);\gamma)=& 
\frac{2\gamma^2-3\gamma-1}{u_3(\gamma)},
\end{align*}
where 
$$
u_3(\gamma)=\gamma(\gamma-1)(\gamma-2)(\gamma+1)(2\gamma+1).
$$
It follows from Theorem \ref{thm:GM} and Theorem \ref{thm:InvGM} that
\begin{align}
&\bE [W_{k_1 k_2} W_{k_3 k_4} W_{k_5 k_6}] \\
=&
\beta^3 \sigma_{k_1 k_2} \sigma_{k_3 k_4} \sigma_{k_5 k_6} 
+ \frac{\beta^2}{2} ( \sigma_{k_1 k_3} \sigma_{k_2 k_4} \sigma_{k_5 k_6}+ 
\sigma_{k_1 k_4} \sigma_{k_2 k_3} \sigma_{k_5 k_6}
+ \sigma_{k_1 k_5} \sigma_{k_2 k_6} \sigma_{k_3 k_4} \notag\\
& \quad + \sigma_{k_1 k_6} \sigma_{k_2 k_5} \sigma_{k_3 k_4} 
+\sigma_{k_1 k_2} \sigma_{k_3 k_5} \sigma_{k_4 k_6} 
+\sigma_{k_1 k_2} \sigma_{k_3 k_6} \sigma_{k_4 k_5}) \notag \\
&\quad + \frac{\beta}{4}(
\sigma_{k_1 k_4}\sigma_{k_2 k_5} \sigma_{k_3 k_6}  + 
\sigma_{k_1 k_3}  \sigma_{k_2 k_5}\sigma_{k_4 k_6} +
\sigma_{k_1 k_4} \sigma_{k_2 k_6} \sigma_{k_3 k_5} +
\sigma_{k_1 k_3} \sigma_{k_2 k_6}\sigma_{k_4 k_5} \notag \\
&\quad +
\sigma_{k_1 k_6} \sigma_{k_2 k_3}\sigma_{k_4 k_5}  +
\sigma_{k_1 k_5} \sigma_{k_2 k_3}\sigma_{k_4 k_6}  +
\sigma_{k_1 k_6} \sigma_{k_2 k_4}\sigma_{k_3 k_5}  +
\sigma_{k_1 k_5} \sigma_{k_2 k_4} \sigma_{k_3 k_6}  
) \notag
\end{align}
and 
\begin{align}
&\bE [W^{k_1 k_2} W^{k_3 k_4} W^{k_5 k_6}] \\
=&u_3(\gamma)^{-1}
\Big[(2\gamma^2-3\gamma-1) \sigma^{k_1 k_2} \sigma^{k_3 k_4} \sigma^{k_5 k_6} \notag \\
&\quad + (\gamma-1) ( \sigma^{k_1 k_3} \sigma^{k_2 k_4} \sigma^{k_5 k_6}+ 
\sigma^{k_1 k_4} \sigma^{k_2 k_3} \sigma^{k_5 k_6}
+ \sigma^{k_1 k_5} \sigma^{k_2 k_6} \sigma^{k_3 k_4} \notag \\
& \quad\quad + \sigma^{k_1 k_6} \sigma^{k_2 k_5} \sigma^{k_3 k_4} 
+\sigma^{k_1 k_2} \sigma^{k_3 k_5} \sigma^{k_4 k_6} 
+\sigma^{k_1 k_2} \sigma^{k_3 k_6} \sigma^{k_4 k_5}) \notag \\
&\quad + (
\sigma^{k_1 k_4} \sigma^{k_2 k_5} \sigma^{k_3 k_6}  + 
\sigma^{k_1 k_3} \sigma^{k_2 k_5} \sigma^{k_4 k_6} +
\sigma^{k_1 k_4} \sigma^{k_2 k_6} \sigma^{k_3 k_5} +
\sigma^{k_1 k_3} \sigma^{k_2 k_6} \sigma^{k_4 k_5} \notag \\
&\quad\quad +
\sigma^{k_1 k_6} \sigma^{k_2 k_3} \sigma^{k_4 k_5}  +
\sigma^{k_1 k_5} \sigma^{k_2 k_3} \sigma^{k_4 k_6}  +
\sigma^{k_1 k_6} \sigma^{k_2 k_4} \sigma^{k_3 k_5}  +
\sigma^{k_1 k_5} \sigma^{k_2 k_4} \sigma^{k_3 k_6}  
)
\Big]. \notag
\end{align}

From Corollary \ref{cor:Trace} we have
\begin{align}
\bE[\bm{p}_\mu(W)]=& \frac{16}{5} 
\( \frac{1}{6} A(\mu,(3)) \bm{p}_{(3)}(\sigma) +\frac{1}{8} A(\mu,(2,1)) \bm{p}_{(2,1)}(\sigma)
+\frac{1}{48} A(\mu,(1^3)) \bm{p}_{(1^3)} (\sigma) \), \\
\bE[\bm{p}_\mu(W^{-1})]=& \frac{16}{5} 
\( \frac{1}{6} B(\mu,(3)) \bm{p}_{(3)}(\sigma^{-1}) 
+\frac{1}{8} B(\mu,(2,1)) \bm{p}_{(2,1)}(\sigma^{-1})
+\frac{1}{48} B(\mu,(1^3)) \bm{p}_{(1^3)}(\sigma^{-1}) \),
\end{align}
for each $\mu \vdash 3$, where 
$$
A(\mu,\rho)= \frac{1}{8} \sum_{\lambda \vdash 3} C_{\lambda}(2\beta) f^{2\lambda} \omega^{\lambda}_\mu
\omega^{\lambda}_{\rho}
\qquad \text{and} \qquad
B(\mu,\rho)= - 8 \sum_{\lambda \vdash 3} C_{\lambda}(-2\gamma)^{-1} f^{2\lambda} \omega^{\lambda}_\mu \omega^{\lambda}_{\rho}.
$$
We compute the matrices 
$A=(A(\mu,\rho))_{\mu,\rho \vdash 3}$ and $B=(B(\mu,\rho))_{\mu,\rho \vdash 3}$.
Here indices of rows and columns of the matrices are labeled 
by $(3), \ (2,1), \ (1^3)$ in order.
By using results in \cite[VII.2]{Mac}, we have
$$
Z := (\omega^\lambda_\mu)_{\lambda,\mu \vdash 3}= \begin{pmatrix}
1 & 1 & 1 \\ -\frac{1}{4} & \frac{1}{6} & 1 \\ \frac{1}{4} & -\frac{1}{2} & 1 \end{pmatrix}.
$$
Since $f^{2\lambda}$ coincides with the number of standard Young tableaux of shape $2\lambda$
(see, e.g., [Sa]\footnote{[Sa] 
B. E. Sagan, 
The symmetric group. Representations, combinatorial algorithms, and symmetric functions,
second ed., Graduate Texts in Mathematics, {\bf 203}. Springer-Verlag, New York, 2001.
}
), we may have
$$
f^{2 (3)}=f^{(6)}=1, \qquad f^{2 (2,1)}=f^{(4,2)}=9, \qquad 
\text{and} \qquad f^{2(1^3)}=f^{(2^3)}=5.
$$
From the definition of $C_\lambda(z)$, it is immediate to see
$$
C_{(3)}(z)= z(z+2)(z+4), \qquad C_{(2,1)}(z)=z(z+2)(z-1), \qquad
\text{and} \qquad C_{(1^3)}(z)=z(z-1)(z-2).
$$
Now, letting $F:= \diag (f^{2(3)}, f^{2(2,1)}, f^{2(1^3)})$ and
$C(z):= \diag (C_{(3)}(z), C_{(2,1)}(z), C_{(1^3)}(z))$, 
we can calculate
$$
A = \frac{1}{8} Z^t \cdot F \cdot C(2\beta) \cdot Z  = \begin{pmatrix}
\frac{15}{16}\beta(2\beta^2+3\beta+2) & \frac{15}{8} \beta(2\beta+1) & \frac{15}{4} \beta \\
  \frac{15}{8}\beta (2\beta+1) & \frac{5}{4} \beta(2\beta^2+\beta+2) & \frac{15}{2} \beta^2 \\
  \frac{15}{4} \beta & \frac{15}{2} \beta^2 & 15 \beta^3
  \end{pmatrix},
$$ 
and
$$
B= -8 Z^t \cdot F \cdot C(-2\gamma)^{-1} \cdot Z 
 = \frac{1}{u_3(\gamma)}
 \begin{pmatrix} \frac{15}{4} \gamma ^2 & \frac{15}{2} \gamma & 15 \\ 
 \frac{15}{2} \gamma & 5 (\gamma^2-\gamma+1) & 15 (\gamma-1) \\
 15 & 15 (\gamma-1) & 15 (2\gamma^2 -3\gamma-1) \end{pmatrix}.
$$
Hence
\begin{align}
\bE[ \bm{p}_{(3)}(W)]=& \frac{1}{2}\beta(2\beta^2+3\beta+2) \bm{p}_{(3)}(\sigma) 
+ \frac{3}{4} \beta(2\beta+1) \bm{p}_{(2,1)}(\sigma)+ \frac{1}{4}\beta \bm{p}_{(1^3)}(\sigma), \\
\bE[ \bm{p}_{(2,1)}(W)]=&  \beta(2\beta+1) \bm{p}_{(3)}(\sigma) 
+ \frac{1}{2} \beta(2\beta^2+\beta+2) \bm{p}_{(2,1)}(\sigma)+ \frac{1}{2} \beta^2 \bm{p}_{(1^3)}(\sigma), \\
\bE[ \bm{p}_{(1^3)}(W)] =& 2\beta \bm{p}_{(3)}(\sigma) 
+ 3\beta^2\bm{p}_{(2,1)}(\sigma)+  \beta^3 \bm{p}_{(1^3)}(\sigma), 
\end{align}
and
\begin{align}
\bE[\bm{p}_{(3)}(W^{-1})]=& 
\frac{ 2\gamma^2\bm{p}_{(3)}(\sigma^{-1}) + 3\gamma \bm{p}_{(2,1)}(\sigma^{-1})
+ \bm{p}_{(1^3)}(\sigma^{-1})}
{\gamma(\gamma-1)(\gamma-2)(\gamma+1)(2\gamma+1)}, \\
\bE[\bm{p}_{(2,1)}(W^{-1})]=& 
\frac{4\gamma \bm{p}_{(3)}(\sigma^{-1}) + 2(\gamma^2-\gamma+1) \bm{p}_{(2,1)}(\sigma^{-1})
+(\gamma-1) \bm{p}_{(1^3)}(\sigma^{-1})}
{\gamma(\gamma-1)(\gamma-2)(\gamma+1)(2\gamma+1)}, \\
\bE[\bm{p}_{(1^3)}(W^{-1})]=& 
\frac{8 \bm{p}_{(3)}(\sigma^{-1}) +6 (\gamma-1) \bm{p}_{(2,1)}(\sigma^{-1})
+(2\gamma^2-3\gamma-1) \bm{p}_{(1^3)}(\sigma^{-1})}
{\gamma(\gamma-1)(\gamma-2)(\gamma+1)(2\gamma+1)}.
\end{align}
We remark that those formulas for $\bE[\bm{p}_\mu(W)]$ \ $(\mu \vdash 3)$
are seen  in \cite[equation (37)]{LM1}.

\subsection*{Degree $4$ and higher degrees}

First we note that,
when $n=4$,
the sums in Theorem \ref{thm:GM}, \ref{thm:InvGM} and \ref{thm:Trace1}
are over $|\mcal{M}(8)|=7 \cdot 5 \cdot 3 \cdot 1=105$ terms.

Consider Corollary \ref{cor:Trace} for any degree $n$. 
As we did in the degree 3 case, we 
can apply it to any degree $n$. The $f^{2\lambda}$ may be computed by 
the well-known hook formula, see e.g. [Sa, Theorem 3.10.2], and 
the $C_{\lambda}(z)$ may be done easily by the definition \eqref{eq:secialC}.
The $\omega^{\lambda}$ are the most complicated among quantities appearing in 
Corollary \ref{cor:Trace}
but we can know their explicit values 
from the table of zonal polynomials in [PJ]\footnote{
[PJ] A. M. Parkhurst and A. T. James,
Zonal polynomials of order $1$ through 12,
Selected Tables in Mathematical Statistics (1974), vol. 2, 199--388.}.

In closing, we give the explicit expressions of Corollary \ref{cor:Trace2} for $n=4$.
Its first formula is given
\begin{equation}
\bE[ (\tr W)^4] = 6 \beta \bm{p}_{(4)}(\sigma)+ 8 \beta^2 \bm{p}_{(3,1)}(\sigma) 
+ 3 \beta^2 \bm{p}_{(2^2)}(\sigma) + 6 \beta^3 \bm{p}_{(2,1^2)}(\sigma)
+  \beta^4 \bm{p}_{(1^4)}(\sigma).
\end{equation}
Suppose $\gamma>0$ but $\gamma \not=\frac{1}{2},1,2,3$.
Put
$$
u_4(\gamma)=\gamma(\gamma-1)(\gamma-2)(\gamma-3)(2\gamma-1)(\gamma+1)(2\gamma+1)(2\gamma+3),
$$
which is non-zero.
From \eqref{eq:tWg} and a list in \cite{CM} (see also \cite{CS}),
we have the explicit values
\begin{align*}
\tWg((4);\gamma)=& \frac{5\gamma-3}{u_4(\gamma)}, &
\tWg((3,1);\gamma)=& \frac{4\gamma(\gamma-2)}{u_4(\gamma)}, \\ 
\tWg((2^2);\gamma)=& \frac{2\gamma^2-5\gamma+9}{u_4(\gamma)}, &
\tWg((2,1^2);\gamma)=& \frac{4\gamma^3-12\gamma^2+3\gamma+3}{u_4(\gamma)}, \\
\tWg((1^4);\gamma)=& \frac{(\gamma+1)(2\gamma-3)(4\gamma^2-12\gamma+1)}{u_4(\gamma)}.
\end{align*}
Hence
the second formula of Corollary \ref{cor:Trace2} at $n=4$ is given
\begin{align}
u_4(\gamma) \cdot \bE[(\tr W^{-1})^4]=& 48(5\gamma-3) \bm{p}_{(4)}(\sigma^{-1}) 
+ 128\gamma(\gamma-2)\bm{p}_{(3,1)}(\sigma^{-1}) \\
&+ 12 (2\gamma^2-5\gamma+9) \bm{p}_{(2^2)}(\sigma^{-1}) \notag \\
&\quad +12 (4\gamma^3-12\gamma^2+3\gamma+3)  \bm{p}_{(2,1^2)}(\sigma^{-1})  \notag \\
&\qquad + (\gamma+1)(2\gamma-3)(4\gamma^2-12\gamma+1) \bm{p}_{(1^4)}(\sigma^{-1}). \notag
\end{align}


\end{document}